\documentclass[a4paper]{amsart}

\usepackage{amsthm,enumerate,verbatim,amssymb,xcolor}
\def\AM{\textit{AM}}
\def\bnu{\boldsymbol{\eta}}
\def\bpi{\boldsymbol{\pi}}
\def\C{\operatorname {Ct}}

\def\cgamma{\operatorname{cap}}

\def\Cx{A}
\def\taux{\tau_A}
\def\A{\mathcal A}

\def\E{\mathcal E}
\def\F{\mathcal F}
\def\G{\mathcal G}

\def\K{\mathcal K}

\def\M{\mathcal M}

\def\ep{\varepsilon}
\def\ff{\varphi}
\def\er{\mathbb R}

\def\rn{\mathbb R^n}

\def\en{\mathbb N}
\def\ep{\varepsilon}

\def\ff{\varphi}

\def\er{\mathbb R}
\def\meas{\boldsymbol{m}}

\def\rn{\mathbb R^n}

\newdimen\vintkern
\def\vint{{-}\kern-\vintkern\int}
\def\eqn#1$$#2$${\begin{equation}\label#1#2\end{equation}}

\newtheorem{thm}{Theorem}[section]
\newtheorem{lemma}[thm]{Lemma}
\newtheorem{cor}[thm]{Corollary}
\newtheorem{prop}[thm]{Proposition}

\theoremstyle{definition}
\newtheorem{definition}[thm]{Definition}
\newtheorem{example}[thm]{Example}

\newtheorem{rmrk}[thm]{Remark}

\title{Plans on measures and $AM$-modulus}
\author{Vendula Honzlov\'a Exnerov\'a}
\address{Department of Mathematical Analysis, Faculty of Mathematics and Physics,
Charles University in Prague, So\-ko\-lovsk\'a 83,  Prague 8, 186\,75 Czech Republic}
\email{venda.ex@gmail.com}

\author{Ond\v{r}ej F.\,K. Kalenda}
\address{Department of Mathematical Analysis, Faculty of Mathematics and Physics,
Charles University in Prague, So\-ko\-lovsk\'a 83,  Prague 8, 186\,75 Czech Republic}
\email{kalenda@karlin.mff.cuni.cz}

\author{Jan Mal\'y}
\address{Department of Mathematical Analysis, Faculty of Mathematics and Physics,
Charles University in Prague, So\-ko\-lovsk\'a 83,  Prague 8, 186\,75 Czech Republic}
\email{maly@karlin.mff.cuni.cz}

\author{Olli Martio}
\address{Department of Mathematics and Statistics, FI-00014 University of Helsinki, Finland}
\email{olli.martio@helsinki.fi}

\thanks{The first and the third authors have been supported by the grant GA\,\v{C}R 
P201/18-07996S of the Czech Science Foundation.
The second author has been suported by the grant GA\,\v{C}R 17-00941S 
of the Czech Science Foundation.} 

\begin{document}
\begin{abstract}%
For measuring families of curves, or, more generally, 
of measures, $M_p$-modulus is traditionally used. More recent studies
use so-called plans on measures.
In their fundamental paper \cite{ADS}, Ambrosio, Di Marino and 
Savar\'e proved that these two approaches are in some sense equivalent
within $1<p<\infty$. We consider the limiting case $p=1$ and show 
that the $AM$-modulus can be obtained alternatively by the plan 
approach. On the way, we demonstrate unexpected behavior of the 
$AM$-modulus in comparison with usual capacities 
and consider the relations between the $M_1$--modulus and the $AM$--modulus.
\end{abstract}

\maketitle

\section{Introduction}
The concept of modulus
of curve family has been introduced by Ahlfors and Beurling \cite{AB},
substantially developed by Fuglede \cite{Fug} and thoroughly exploited
in geometric function theory, see \cite{MRSY} for an overview.
It can be shown
that the $W^{1,p}$-capacity of a set $A$ can be computed as the $p$-modulus of a 
special family of curves (Ziemer \cite{Zie}).
The modulus is an outer measure on the family of all 
rectifiable curves in a space and can be 
used to determine ``small families'' of curves, for example, the family
of all curves along which a $W^{1,p}$ function fails to be absolutely 
continuous,
see \cite{Fug}.

A parallel way of measuring families of curves is based on so-called 
\textit{plans}.
This concept has been introduced by Ambrosio,  Gigli and Savar\'e
\cite{AGS1}, \cite{AGS}.
Their research is motivated by applications to 
PDEs of the first order, gradient flows, heat flows, measure 
transportation
and to analysis  in metric 
measure spaces, in particular, to function spaces of the first
order.

A special attention must be paid if we are interested in $p=1$.
The $1$-modulus can be applied in connection with the Sobolev spaces
$W^{1,1}$, whereas it does not fit well if we are interested in $BV$ spaces.
Recently, Martio introduced the $AM$-modulus, which corresponds well to
the $BV$ theory. On the other hand, the $1$-plans can be well applied to the 
$BV$-theory, see Ambrosio and Di Marino \cite{ADM}.

Our research is motivated by the paper \cite{ADS} by Ambrosio,
Di Marino and Savar\'e, where it is shown that the $p$-modulus
and $p$-plan-content lead to the same result if $p>1$. 
We compare the $AM$-modulus and
the $1$--plan content introduced in Section~\ref{sec:plans}. 
Both our results and results of \cite{ADS}
work in a more general framework of families of measures (instead of curves).

The $AM$--modulus, the $M_1$--modulus, and the $1$--plan content are more
sensitive to the topology of the positive cone $\M^+(X)$ of finite
Radon measures in $X$ than the $M_p$--modulus, $p>1$.
Section~\ref{sec:topologies} is devoted to preparations to this analysis.
We consider a family of induced weak* topologies on $\M^+(X)$, so called $\tau_A$-topologies, including 
the $\tau_0$-topology from $C_0(X)^*$ if $X$ is locally compact and 
the $\tau_b$-topology from $C_b(X)^*$.

We study capacity-like properties of the $AM$--modulus, the $M_1$--modulus, and the $1$--plan content.
In many respects, both the $M_1$-- and the $AM$--modulus behave quite differently from the $M_p$--modulus, $p > 1$. 
Contrary to $p > 1$ neither the $M_1$-- nor the $AM$--modulus define a Choquet
capacity in $\M^+(X)$.
We show that the $AM$--modulus
is continuous under (all) increasing families of measures in $\M^+(X)$ only in trivial cases and, 
moreover, we provide a counterexample consisting of families of curves. In spite of this fact,
the  $M_1$-- and the $AM$--modulus  have many capacity like properties which are studied
in Section~\ref{sec:moduli} together with the interplay of the $M_1$-- and the $AM$--modulus. 
We employ these properties to show the $M_1$--modulus is not a Choquet capacity.
On the other hand, the $1$-plan content is 
continuous with respect to increasing families Choquet capacity, as shown in Section~\ref{sec:plans}. 

The fact that continuity with respect to increasing families is a subtle property can be demonstrated
on the Hausdorff content on a metric space $Y$, for which it holds. 
Already Federer's proof under the assumption that the underlying space is boundedly
compact is difficult, see \cite[Thm. 2.10.22]{Fed}. The general case is very deep and due to Davies 
\cite[Thm. 8]{Dav}.

Another property required in the definition of Choquet capacity is continuity with respect to decreasing
families of compact sets. In Proposition \ref{L:wupper} 
we show that the $AM$--modulus, the $M_1$--modulus, and the $1$--plan content
satisfy this property on $(\M^+(X),\tau_A)$. For the $1$--plan content this completes the verification
that it is a Choquet capacity. In fact, the continuity with respect to decreasing
families of compact sets of these functions is obtained for all $1\le p<\infty$. For $\tau_b$
the range $p>1$ follows already from \cite{ADS}, for other topologies it is new.

In Section~\ref{sec:m-c} we start comparison of 
$AM$ and the $1$-plan--content.
It is shown that
the $M_1$-- and the $AM$--modulus, as well as the $1$--plan content, coincide on $\tau_A$-compact families.
This is a partial counterpart of the equality result in \cite{ADS}. However,
the passing from compact families to general families is not as straightforward as for $p>1$.
The reason consists in the above mentioned failure of continuity on increasing families.
Nevertheless, in Section~\ref{sec:lc} a connection between
the $AM$--modulus and a special content in locally 
compact Polish spaces $X$ is obtained. (Note that local compactness is satisfied in most applications.)
Namely, in Theorem \ref{t:main} we show that, for an arbitrary family $\E\subset\M^+(X)$,
$$
AM(\E)=\inf\{\C(\G)\colon \G\supset \E \text{ is }\F_{\sigma} \text{ in }\tau_0\},
$$
where $\C$ is the $1$--plan content. As the expression on the right hand side can be regarded
as a natural outer-capacity-like function on families of measures created from plans,
we achieve the goal that the modulus approach and plan approach meet on arbitrary families of measures
also for $p=1$.

\section{Choquet capacity theory}\label{s:choquet}

In this section we recall the basic notions of Choquet's capacity theory \cite{Choq}
including the famous Choquet capacitability theorem. 

In next sections, we will study
capacity-like set functions on \textit{sets of measures},
i.e.\ our choice of $Y$ will
$\M^+(X)$ embedded into the positive cone of $C_b(X)^*$,
where $X$ is a metric measure space.

In this setting, Ambrosio,
Di Marino and Savar\'e \cite{ADS} proved that the $p$-modulus is a Choquet capacity,
see Theorem \ref{t:choquet}.

We will demonstrate that the situation for $p=1$ is
not as simple and show that some of properties listed below fail although
one could expect them to hold.

The classical motivation for Choquet capacity theory comes from
the potential theory where the fundamental example is that of
Newtonian capacity.

Let $Y$ be a Hausdorff topological space.

\begin{definition}\label{d:capacity}
Let $\cgamma\colon 2^Y\to[0,\infty]$ be a set function. Let us list some
properties that such functions may (or may not)  have.
\begin{enumerate}[(i)]
\item \label{monot}
For each $A,B\subset Y$,
$$
A\subset B\implies \cgamma (A)\le\cgamma(B)
$$
(monotonicity),
\item \label{upward}
for each $A_1,A_2,\dots \subset Y$,
$$
A_1\subset A_2\subset\dots \implies \cgamma\Bigl(\bigcup_{j=1}^{\infty}A_j\Bigr)
=\lim_{j\to\infty}\cgamma(A_j),
$$
\item \label{downward}
for each compact $K_1,K_2,\dots\subset Y$,
$$
K_1\supset K_2\supset\dots \implies \cgamma\Bigl(\bigcap_{j=1}^{\infty}K_j\Bigr)
=\lim_{j\to\infty}\cgamma(K_j),
$$
\item\label{outer}
For each $E\subset Y$,
$$
\cgamma(E)=\inf(\cgamma(G)\colon G\supset E,\;G\;\text{open})
$$
(outer regularity),
\item\label{wupper}
For each $K\subset Y$ compact,
$$
\cgamma(K)=\inf(\cgamma(G)\colon G\supset K,\;G\;\text{open}),
$$
\item\label{inner}
For each $E\subset Y$,
$$
\cgamma(E)=\sup(\cgamma(K)\colon K\subset E,\;K\;\text{compact})
$$
(inner regularity).
\end{enumerate}
We say that $\cgamma$ is a \textit{Choquet capacity} if it satisfies
\eqref{monot}--\eqref{downward} and $Y$ is a Polish space.

\end{definition}

\begin{rmrk}
If the set function $\cgamma$ satisfies \eqref{monot} and \eqref{wupper},
it also satisfies \eqref{downward}. Conversely, \eqref{monot} and 
\eqref{downward} imply \eqref{wupper} if $Y$ is metrizable and locally compact.
\end{rmrk} 

\begin{rmrk}\label{r:inner}
Set functions that appear in applications are seldom both 
outer and inner regular.
Classical capacities satisfy \eqref{monot}--\eqref{wupper}, but not
\eqref{inner}. For them, the outer regular capacity $\cgamma$ in consideration, has an
``associated'' inner regular ``inner capacity'' $\cgamma_*$,
which satisfies \eqref{monot} and \eqref{inner} and coincide with 
$\cgamma$ at least on Borel sets. Thus, $\cgamma_*$ satisfies the formula from
\eqref{upward} if we restrict our attention to sets on which $\cgamma_*=\cgamma$.
\end{rmrk}

\begin{definition}\label{d:capacitable} 
Let $\cgamma\colon 2^Y\to[0,\infty]$ be a set function
and $E\subset Y$. We say that $E$ is \textit{$\cgamma$-capacitable} if
$$
\cgamma(E)=\sup(\cgamma(K)\colon K\subset E,\;K\;\text{compact}).
$$
Let us emphasize that we use this terminology even if $\cgamma$ is not a (Choquet)
capacity.
\end{definition}

\begin{thm}[Choquet capacitability theorem \cite{Choq}]
Let $\cgamma$ be a Choquet capacity on a Polish space $Y$. Then every
Souslin set $E\subset X$ is $\cgamma$-capacitable.
\end{thm}
\begin{proof} See also e.g.\ \cite[Appendix]{AE} or \cite[Theorem 30.13]{kechris}.
\end{proof}

\begin{prop}[cf. Example 30.B.1 in \cite{kechris}]\label{p:outer} 
Let $\mu$ be a 
locally
finite Borel measure on a Polish space $Y$
and let $\mu^*$ be the induced outer measure, this means
$$
\mu^*(E)=\inf\{\mu(A)\colon A\text{ Borel }, A\supset E\}.
$$
Then $\mu^*$ is a Choquet capacity. In particular,
$$
\mu^*(A)=\sup\{\mu(K)\colon K\subset A\text{ compact}\}
$$
holds for any Souslin set $A$, so that $\mu$ 
is a Radon measure.
\end{prop}

\section{Topologies on measures}\label{sec:topologies}

In this section,  $X$ will be a fixed Polish space. 

First we introduce some terminology in a more general setting.
If $\mu$ is a Borel measure on a Hausdorff completely regular space $Y$, 
we say that $\mu$ is a \textit{Radon measure}
if $\mu$ is finite on compact sets and \textit{inner regular} with respect to 
compact sets.
(Note that there is a discrepancy in literature what should be the meaning of ``Radon measure''.)
By $\M(Y)$ we denote the vector space of all 
finite signed Borel measures on $Y$ and $\M^+(Y)$ stands for the cone formed by positive measures 
from $\M(Y)$. 
In the full generality, we distinguish $\M(Y)$ (all finite signed Borel measures) and $\M(Y,t)$
(finite signed Radon measures, the letter $t$ comes from \textit{tight} which is 
another term used for the inner regularity).
Since $\M(Y,t)$ is canonically embedded to the dual space $C_b(Y)^*$ 
(where $C_b(Y)$ is the space of all bounded continuous functions on $Y$ endowed with the supremum norm), 
it may be considered with the inherited weak$^*$ topology. This topology will 
be denoted by $\tau_b$ in the sequel.

In our setting of Polish spaces, it follows from Proposition \ref{p:outer} that any 
locally finite Borel measure is Radon, in particular, $\M^+(X)=\M^+(X,t)$. 

Conversely, any Radon measure on a metrizable space $X$ is locally finite. 
(Assume that $\mu$ is Radon and not locally finite.
This means that there is some $x\in X$ such that $\mu(U)=\infty$ for 
each open set $U$ containing $x$. Fix a metric $d$ on $X$. Using the 
assumption and inner regularity we find a sequence $(K_n)_n$  of compact 
subsets of $X$ such that $K_n\subset B(x,\frac1n) \setminus(\{x\}\cup \bigcup_{k<n} K_k)$
and $\mu(K_n)>1$. Then 
$\{x\}\cup\bigcup_n K_n$ is a compact set of infinite measure, a contradiction.)

It is now easy to deduce that any Radon measure on a Polish space is outer regular with respect to open sets (cf. \cite[Proposition 412W(b)]{fremlin}). 

The topology $\tau_b$ restricted to $\M^+(Y,t)$ has several useful properties 
summarized in the following lemma.

\begin{lemma}\label{L:taub}
Let $T$ be a Hausdorff completely regular space.
%, let $\M^+_t(T)$ denote the cone of finite positive Radon measures on $T$ and let $\tau_b$ denotes the weak$^*$ topology inherited from $C_b(T)^*$.
Then the following holds.
\begin{itemize}
	\item[(i)] The topology $\tau_b$ is the weakest one in which 
	$\mu\mapsto\mu(T)$ is continuous and $\mu\mapsto\mu(G)$ is 
	lower semi-continuous for each open set $G\subset T$.
	\item[(ii)] If $S\subset T$, then $\M^+(S,t)$ is a topological subspace of $\M^+(T,t)$.
\end{itemize}
\end{lemma}

\begin{proof} Assertion (i) follows from the equivalence (v)$\Leftrightarrow$(viii) in \cite[Theorem 8.1]{Topsoe}.

Assertion (ii) is an easy consequence of (i) as remarked in \cite[Lemma 1]{HoKa}.
\end{proof}

Now, we return to the generality of Polish spaces.

The space $(\M^+(X),\tau_b)$ is Polish (see Proposition~\ref{P:tauA} below).
However, in case $X$ is locally compact there is another natural topology. In this case the space $\M(X)$ (equipped with the total variation norm) is, due to the Riesz theorem, isometric to the dual space $C_0(X)^*$ (where $C_0(X)$ is the space of continuous functions on $X$ vanishing at infinity equipped with the supremum norm which coincides with the closure of continuous functions with compact support in $C_b(X)$. The respective weak$^*$ topology will be denoted by $\tau_0$. By Banach-Alaoglu theorem $\M(X)$ (and hence also $\M^+(X)$) is $\sigma$-compact in $\tau_0$.

If $X$ is even compact, then clearly $\tau_b=\tau_0$. However, if $X$ is not compact, these two topologies are different. In particular, $(\M^+(X),\tau_0)$
is not metrizable and $(\M^+(X),\tau_b)$ is not $\sigma$-compact (see Proposition~\ref{P:tauA} below).

In order to cover these two interesting cases and possibly some other cases as well we adopt an abstract
approach.

Let us call a closed subspace $\Cx(X)$ of $C_b(X)$ \emph{acceptable} if it satisfies the following two properties.
\begin{itemize}
	\item[(A1)] $\Cx(X)$ is a sublattice of $C_b(X)$.
	\item[(A2)] For any closed subset $F\subset X$ and $x\in X\setminus F$ there exists $g\in\Cx(X)$ such that $0\le g\le 1$, $g(x)=1$ and $g|_F=0$.
	\end{itemize}

In the following lemma we collect several equivalent formulations of the property (A2). The property (A4) will the the key one.

\begin{lemma}\label{L:Ax} Let $\Cx(X)$ be a closed linear sublattice of $C_b(X)$. Then the following assertions are equivalent.
\begin{itemize}
	\item[(A2)] For any closed subset $F\subset X$ and $x\in X\setminus F$ there exists $g\in\Cx(X)$ such that $g(x)=1$ and $g|_F=0$.
		\item[(A3)] For any $f\in C_b(X)^+$ we have
		$$f(x)=\sup\{g(x)\colon g\in\Cx(X), 0\le g\le f\},\quad x\in X.$$ 
			\item[(A4)] For any $f\in C_b(X)^+$ there exists a sequence $(g_j)_j$ of functions
from $\Cx(X)^+$ such that $g_j\nearrow f$.
  \item[(A5)] The family 
		$$\{x\in X\colon g(x)>0\},\quad g\in \Cx(X)$$
		form a base of the topology of $X$.
		\end{itemize}
\end{lemma}

\begin{proof}
The equivalence (A2)$\Leftrightarrow$(A5) is obvious.

(A2)$\Rightarrow$(A3): Fix $f\in C_b(X)^+$. It is enough to prove that,
given $x\in X$ and $c>0$ such that $f(x)>c$, there is $g\in A(X)^+$ with $0\le g\le f$ and $g(x)\ge c$.
To this end set
$F=\{y\in X\colon f(y)\le c\}$. Then $F$ is closed and $x\notin F$. Hence there is $h\in \Cx(X)$ with $0\le h\le 1$, $h|_{F}=0$ and $h(x)=1$. The choice $g=c h$ does the job.

(A3)$\Rightarrow$(A4)  Fix $f\in C_b(X)^+$ and set $M=\{g\in A(X)^+; g\le f\}$.
It follows that $f=\sup M$. For $g\in M$ and $n\in\mathbb{N}$ let
$$U(g,n)=\{x\in X\colon  g(x)>f(x)-\frac1n\}.$$
For each $n\in\mathbb{N}$ the family
$$\{U(g,n)\colon g\in M\}$$
is an open cover of $X$. It follows that there is a countable subcover, i.e., a countable set $M_n\subset M$ such that
$$\{U(g,n)\colon g\in M_n\}$$
covers $X$. It follows that
$$f=\sup \bigcup_n M_n.$$
Now enumerate $\bigcup_n M_n=\{h_m\colon m\in\en\}$ and set $g_n=\max\{h_1,\dots,h_n\}$ for $n\in\en$.

(A4)$\Rightarrow$(A2)  Let $F\subset X$ be a closed subset and $x\in X\setminus F$. Then there is $f\in C_b(X)$ such that $0\le f\le 1$, $f(x)=1$, $f|_F=0$. Let $(g_n)$ be a sequence in $\Cx(X)^+$ with $g_n\nearrow f$. Then $g_n|_F=0$ for each $n$ and $g_n(x)\nearrow 1$. Hence, there is some $n$ with $g_n(x)>0$. A suitable multiple of $g_n$ does the job.
\end{proof}

\begin{lemma}\label{L:aprox}
Let $\Cx(X)$ be an acceptable subset of $C_b(X)$. Then the following assertions are true.
\begin{itemize}
	\item[(i)] There exists $h\in \Cx(X)$ such that $h>0$ on $X$.
		\item[(ii)] If $u\colon X\to\er^+$ is lsc, then there is a sequence $(g_k)_k$
of functions from $\Cx(X)^+$ such that $g_k\nearrow u$.
\end{itemize}
\end{lemma}

\begin{proof}
(i)   By (A4) there is a sequence $(h_j)$ in $\Cx(X)^+$ such that $h_j\nearrow1$. Set 
$h=\sum_j2^{-j}h_j$.

(ii) Since $\Cx(X)$ is a lattice, it is enough to find a countable subset $M\subset\Cx(X)^+$ with $\sup M=u$. But this is easy using (A4) and the fact that there is a sequence $(f_k)$ in 
$C_b(X)^+$ with $f_k\nearrow u$ 
(see e.g.\ \cite[Proposition 3.4]{LMZ}).
\end{proof}

Let us collect some examples of acceptable subspaces:

\begin{itemize}
\item $C_b(X)$ itself is acceptable, i.e., we may have $\Cx(X)=C_b(X)$.
\item If $X$ is locally compact, then $\Cx(X)=C_0(X)$ is acceptable.
\item If $\gamma X$ is a compactification of $X$, $\Cx(X)$ can consist of the restrictions to $X$ of functions from $C(\gamma X)$.
In particular, if $X$ is locally compact we may consider its one-point compactification $\alpha X$.
\item
If $X$ is embedded into a Polish space $Y$ as an open subset of $Y$,
$\Cx(X)$ can be 
$$
\{f\in C_b(X)\colon \lim_{x\to z}f(x)=0\text{ for each }z\in\partial X\}.
$$
\item
If the topopogy of $X$ is induced by a  metric $d$ and $x_0\in X$ is a fixed point,
$\Cx(X)$ can be the set $$\{f\in C_b(X)\colon \lim_{d(x,x_0)\to\infty}f(x)=0\}.$$
\item If $X=(0,1)$, we may take
$$\Cx(X)=\{f|_{(0,1)}\colon f\in C([0,1]), f(1)=2f(0)\}.$$
%\item If $X$ is the open unit disc in the plane, we may take
%$$\Cx(X)=\{f|_X\colon f\in C(\overline X), f(0,1)=0, f(1,0)=2f(-1,0).\}$$
\end{itemize}

\begin{rmrk} It is well known and easy to see that any closed subalgebra of $C_b(X)$ is also a sublattice. 
(This is proved in \cite[Lemma 3.2.20]{Eng} under the additional assumption that it contains constant functions.
But this assumption is not necessary as the sequence of polynomials $(w_i)$ chosen in \cite[Lemma 3.2.20]{Eng} 
may be easily adapted to satisfy $w_i(0)=0$ for each $i$.)

 So, any subalgebra satisfying (A2) is an acceptable subspace.
\end{rmrk}

\begin{lemma}\label{L:acceptable}
Let $\Cx(X)$ be an acceptable subspace of $C_b(X)$. Then the following holds.
\begin{itemize}
	\item[(i)] There is a (in general non-metrizable) compactification $K$ of $X$ and a closed linear sublattice $\Cx(K)$ of $C(K)$ separating points of $K$
such that
$$\Cx(X)=\{f|_X\colon f\in \Cx(K)\}.$$
\item[(ii)] If $\Cx(X)$ contains constant functions and $K$ is as in (i), then $\Cx(K)=C(K)$.
\item[(iii)] If $\Cx(X)$ is even a subalgebra of $C_b(X)$ and does not contain constant functions and $K$ is as in (i), then there is some $z\in K\setminus X$ such that
$$\Cx(K)=\{f\in C(K)\colon f(z)=0\}.$$
\item[(iv)] Assume that $\Cx(X)$ does not contain constant functions but there is some $h\in\Cx(X)$ with $\inf_x h(x)>0$.
 Then
$$\Cx_h(X)=\left\{\frac fh\colon f\in \Cx(X)\right\}$$
is an acceptable subspace of $C_b(X)$ containing constant functions.
\end{itemize}
 \end{lemma}

\begin{proof}
(i) This follows by the standard construction of compactifications,
namely
 $K$ is the closure in $\er^{\Cx(X)}$ of the canonical copy of $X$ formed by evaluation functionals.

(ii) This follows from the Stone-Weierstrass theorem.

(iii)  Observe that $\operatorname{span}(\Cx(K)\cup\{1\})$ is a closed subalgebra of $C(K)$ separating points of $K$. Hence, by the Stone-Weierstrass theorem we get that $\operatorname{span}(\Cx(K)\cup\{1\})=C(K)$. It follows that $\Cx(K)$ is a hyperplane in $C(K)$, i.e., it is the kernel of a functional $\varphi\in C(K)^*$. Since $1\notin\Cx(K)$, we may assume without loss of generality that $\varphi(1)=1$. In this case, using the fact that
$\Cx(K)$ is also a subalgebra, it is easy to check that $\varphi$ is multiplicative. Hence $\varphi$ is represented by a Dirac measure $\delta_z$ for some $z\in K$ (cf. \cite[Example 11.13(a)]{rudin})  . It follows from (A2) that $z\in K\setminus X$.

(iv) Note that the mapping $f\mapsto \frac fh$ is an isomorphism of $C_b(X)$ onto itself which is also a lattice homomorphism. It follows that $\Cx_h(X)$ is a closed linear sublattice of $C_b(X)$. Moreover, the validity of condition (A2) is clearly preserved.
\end{proof}

If $\Cx(X)$ is an acceptable subspace of $C_b(X)$, then, in particular, it separates points of $X$.
Hence $\M(X)$ can be considered as a subspace of $\Cx(X)^*$,
so we can equip $\M(X)$ with the restricted weak* topology of $\Cx(X)^*$,
which will be denoted by $\taux$. 
The above-defined topologies $\tau_b$ and $\tau_0$ are special cases.
We restrict the topologies $\taux$ to the positive cone
$\M^+(X)$ of $\M(X)$. 
The duality pairing of a measure $\mu\in\M(X)$ and a function $g\in \Cx(X)$
is denoted by $\langle\mu,g\rangle$.
The following proposition summarizes basic properties of these topologies.

\begin{prop}\label{P:tauA} 
Let $X$ be a Polish space and $\Cx(X)$ be an acceptable subspace of $C_b(X)$.
Let $P(X)$ be the set of probability measures on $X$.
Then the following assertions hold.
\begin{itemize}
	\item[(a)] $(\M^+(X),\tau_b)$ is a Polish space. It is $\sigma$-compact if and only if $X$ is compact.
	\item[(b)] If $\Cx(X)$ contains constant functions, then $\taux=\tau_b$ on $\M^+(X)$.
	\item[(c)] If $\Cx(X)$ contains a function $h$ with $\inf h(X)>0$, then $\taux=\tau_b$ on $\M^+(X)$.
	\item[(d)] $\taux=\tau_b$ on $P(X)$.
	\item[(e)]  $(\M^+(X),\taux)$ is metrizable if and only if there is a function $h\in\Cx(X)^+$ with $\inf h(X)>0$.
	\item[(f)] Any $\taux$-compact subset of $\M^+(X)$ is metrizable.
	\item[(g)] If $X$ is locally compact, then $(\M^+(X),\tau_0)$ is $\sigma$-compact.
		Moreover, $\{\mu\in \M^+(X)\colon \mu(X)\le c\}$ is metrizable in $\tau_0$ for each $c>0$.
		
		The whole space $\M^+(X)$ is metrizable in $\tau_0$ if and only if $X$ is compact.
	\item[(h)] Any $\tau_b$-open subset of $(\M^+(X),\tau_b)$ is $\taux$-$\F_\sigma$. Hence, $\tau_b$-Borel sets and $\taux$-Borel sets in $\M^+(X)$ coincide.
\end{itemize}
\end{prop}

\begin{proof}
(a) The first statement is well known (see, e.g., \cite[Appendix, Theorem 7]{Schwartz} 
or \cite[Corollary (a)]{HoKa}). 

If $X$ is compact, then $(M_+(X),\tau_b)$ is $\sigma$-compact by Banach-Alaoglu theorem.

If $X$ is not compact, then there is a sequence $(x_n)$ in $X$ with no cluster point. Then $F=\{x_n\colon n\in\en\}$ is a closed subset of $X$, thus 
$\M^+(F)$ is a closed subset of $\M^+(X)$ (this follows easily from Lemma~\ref{L:taub}(i)). By Lemma~\ref{L:taub}(ii) the subspace topology on $\M_+(F)$ coincide with its $\tau_b$-topology. Since $F$ is a countable discrete set, $(\M^+(F),\tau_b)$ is homeomorphic to the positive cone $\ell_1^+$ of the Banach space $\ell_1$ equipped with the weak topology (note that $C_b(\en)=\ell_\infty=\ell_1^*$). This cone is not $\sigma$-compact as, otherwise $\ell_1=\ell_1^+-\ell_1^+$ would be also weakly $\sigma$-compact and hence $\ell_1$ would be reflexive, which is not the case.

(b) Assume $\Cx(X)$ contains constant functions.
 It follows from Lemma~\ref{L:acceptable}(ii) that there is a compactification $K$ of $X$ such that on $\M^+(X)$ the topology $\taux$ coincides with the weak$^*$ topology inherited from $C(K)^*$. But this one coincides with $\tau_b$ by Lemma~\ref{L:taub}.

(c) It easily follows from Lemma~\ref{L:aprox}(ii) that $\mu\mapsto\mu(U)$ is $\taux$-lower semicontinuous for each open set $U\subset X$. Further, by the assumption there is $h\in \Cx(X)$ with $h\ge 1$.  It follows from (A4) that there is a sequence $(g_n)$ in $\Cx(X)^+$ with $g_n\nearrow h-1$. Hence
$$ h\ge h-g_n\searrow 1.$$
By the monotone convergence theorem we get 
$$\langle \mu,h-g_n\rangle\searrow \langle \mu,1\rangle =\mu(X).$$
Thus the function $\mu\mapsto \mu(X)$ is $\taux$-upper semi-continuous. Since $X$ is also an open set, we deduce that $\mu\mapsto \mu(X)$ is $\taux$-continuous.

Hence, using Lemma~\ref{L:taub}(i) we see that $\taux$ is finer than $\tau_b$ on $\M^+(X)$. Since it is clearly simultaneously weaker that $\tau_b$, we conclude that $\taux=\tau_b$ on $\M^+(X)$.

(d) It is clear that $\taux$ is weaker than $\tau_b$. Conversely, since $\mu(X)=1$ for each $\mu\in P(X)$, it follows from  Lemma~\ref{L:taub}(i) that the sets
$$\{\mu\in P(X)\colon \mu(U)>c\},\quad c\in\er, U\subset X\mbox{ open},$$
form a subbase of the topology $\tau_b$ restricted to $P(X)$. But these sets are also $\taux$-open as, in the same way as in the proof of (c), we deduce from Lemma~\ref{L:aprox}(ii) that $\mu\mapsto\mu(U)$ is $\taux$-lower semicontinuous for each open set $U\subset X$. This completes the proof.

(e) The `if part' follows from assertions (c) and (e). The `only if part' will be proved by contradiction. To this end assume that $\inf f(X)=0$ for each $f\in \Cx(X)^+$ but $(\M^+(X),\taux)$ is metrizable. Then there exist a countable base of neighborhoods of zero. Hence we may find functions
$0\le f_1\le f_2\le \dots$ in $\Cx(X)$ such that 
$$
V_k:=\{\mu\in\M^+(X)\colon \mu(f_k)<1\},\quad k\in\en
$$
form a base of neighborhoods of the origin.

Fix $k\in\en$. Since $\inf f_k(X)=0$, we find $x_k\in X$ with $f_k(x_k)<2^{-k}$.
Applying condition (A3) to the constant function $2^{-k}$ we find $g_k\in\Cx(X)$ with
$0\le g_k\le 2^{-k}$ with $g_k(x_k)>f_k(x_k)$.
 
	Set 
$$
g=\sum_{k=1}^\infty g_{k}.
$$
Then $g\in \Cx(X)^+$ and thus 
$$
V:=\{\mu\in\M^+(X)\colon \langle\mu,g\rangle<1\}
$$
is a $\taux$-neighborhood of $0$.

Given $k\in\en$, we may choose $a_k\in\er$ such that $f_k(x_k)<a_k<g(x_k)$.
Then
$\delta_{x_k}/a_k\in V_k\setminus V$.
Hence $V_k\setminus V\ne \emptyset$ for each $k\in\en$,
a contradiction. 

(f) Since $(\M^+(X),\tau_b)$ is a Polish space, its topology has a countable base. Since $\taux$ is weaker than $\tau_b$, the base of $\tau_b$ is a network for $\taux$ (cf. \cite[p. 127]{Eng}). It is well known that any compact space with countable network has even a countable base (see \cite[Theorem 3.1.19]{Eng}) an hence it is metrizable by \cite[Theorem 4.2.8]{Eng}.

(g) The compactness of $\{\mu\in\M^+(X)\colon\mu(X)\le c\}$ follows from Banach-Alaoglu theorem. 
This set is metrizable by assertion (f) and $\sigma$-compactness of the whole space is obvious.

If $X$ is compact, then $\tau_0=\tau_b$, hence $(\M^+(X),\tau_0)$ is metrizable by assertion (a). If $X$ is not compact, the non-metrizability follows from assertion (e).

(h) A base of $(M^+(X),\tau_b)$ is formed by sets of type 
$$
U_{\mu_0,g}:=\{\mu\in M^+(X)\colon |\langle\mu-\mu_0,\;g\rangle| <1\},\qquad \mu_0\in M^+(X),\;g\in C_b(X)^+.
$$
Fix $\mu_0$ and $g$ as above 
and use property (A4) to find a sequence $g_j\in \Cx(X)^+$ such that $g_j\nearrow g$.
Then 
$$
U_{\mu_0,g}=\bigcup _{k=1}^{\infty}\bigcap_{j\ge k}\Big\{\mu\colon |\langle\mu-\mu_0,\;g_j\rangle| \le 1-\frac1k\Big\},
$$ 
so $U_{\mu_0,g}$ is $\taux$-$\F_\sigma$. Since $(\M^+(X),\tau_b)$ is a Polish space, any open set is a countable union of basic open sets, hence it is also $\taux$-$F_\sigma$.

It follows that any $\tau_b$-Borel set is $\taux$-Borel. The converse implication follows from the fact that the topology $\taux$ is weaker than $\tau_b$.
\end{proof}

\section{Moduli  of families of measures}\label{sec:moduli}

In this section $X$ will be again a Polish space  and $\Cx(X)$ a fixed acceptable subspace of $C_b(X)$. In addition, we assume that $X$ is equipped with a \textit{reference measure} $\meas$
which is a non-negative Radon measure (not necessarily finite). 

\begin{definition}
Let $\E$ be a  subfamily of $\M^+(X)$. 
We say that a lsc function $\rho\colon X\to [0,\infty]$ 
is an \textit{admissible function} for $\E$ if 
$$
\int_X\rho\,d\mu \ge 1\text{ for each }\mu\in \E.
$$
Let $p\in [1,\infty)$. 
%]$
We define the \textit{$L^{p}$-modulus} of $\E$ 
and its continuous version as 
$$
\aligned
M_p(\Gamma)&=\inf \Bigl\{\|\rho\|_{L^{p}(\meas)}^p\colon \rho \text{ is 
admissible for }\E\Bigr\},\\
M_{p,c}^A(\Gamma)&=\inf \Bigl\{\|\rho\|_{L^{p}(\meas)}^p\colon 
\rho\in \Cx(X) \text{ is admissible for }\E\Bigr\}.
\endaligned
$$
\end{definition}

\begin{definition}[\cite{M1}]
Let $\E$ be a  subfamily of $\M^+(X)$. 
We say that a sequence of lsc functions $(\rho_j)_j$, $\rho_j\colon X\to [0,\infty]$ 
is an \textit{admissible sequence} for $\E$ if 
$$
\liminf_j\int_X\rho_j\,d\mu \ge 1\text{ for each }\mu\in \E.
$$
Let $p\in [1,\infty)$. 
%]$
We define the \textit{$L^p$-approximation 
modulus} of $\E$ and its continuous version as 
$$
\aligned
AM_p(\E)&=\inf \Bigl\{\liminf_{j}\|\rho_j\|_{L^p(\meas)}^p\colon (\rho_j)_j \text{ 
is admissible for }\E\Bigr\},\\
AM_{p,c}^A(\E)&=\inf \Bigl\{\liminf_{j}\|\rho_j\|_{L^p(\meas)}^p\colon 
 (\rho_j)_j\,\subset\,\Cx(X) 
 \text{ 
is admissible for }\E
\Bigr\}.
\endaligned
$$
If $p=1$, we simplify $AM_{1}$ to $AM$ and $AM_{1,c}^A$ to $AM_c^A$.
\end{definition}

\begin{definition}
Most frequently, the concept of modulus is used on families
of paths. If we speak on paths, the topology on $X$ is induced by a fixed metric, so
that the length of a curve  $\ff\colon [a,b]\to X$ makes sense. A path $\ff\colon [a,b]\to X$ is a
non-constant curve of finite length. Then $\ff$ can be parametrized by arc length,
the reparametrization it is of the form $\gamma=\ff\circ h$, where $h\colon [0,\ell]\to[a,b]$ 
is a suitable increasing homeomorphism and $\ell$ is the total length of the path $\ff$.
The curvelinear integral of a function $\rho\colon X\to\er$ over $\ff$ is then defined as
$$
\int_{\ff}\rho\,ds=\int_0^{\ell}\rho(\gamma(t))\,dt
$$
whenever the integral on the right makes sense. 
A path $\ff$ induces a finite
Radon measure $\mu_{\ff}$ defined as
$$
\mu_{\ff}(E)=\int_{\ff}\chi_{E}\,ds,\qquad E\subset X \text{ Borel}.
$$
It is the push-forward of the Lebesgue measure under the arc-lenght reparametrization of $\ff$. 
The modulus of a family of paths is then defined as the modulus of the family of induced measures.
Note that one Radon measure can be induced by different paths even if we insist on arc-length parametrization,
as we do not require the paths being injective. By definition, the modulus treats only the resulting measures.
\end{definition}

\begin{rmrk}
 Due to the motivation, we point out if
an example of any phenomena can be demonstrated on paths.
However, note that even when dealing with paths we use 
the $\taux$-topology
on $\M^+(X)$, 
which differs from
topologies used usually on families of paths, see
e.g.\ \cite{Sav}. In a correspondence with \cite{ADS}
we prefer an approach based on the topology on measures.
It would be interesting to develop a parallel research
on paths and discuss the dependence of situation on the choice of topologies.
\end{rmrk}

\begin{rmrk}  
The $AM$-modulus 
on paths
is related to $BV$ spaces, see \cite{HEMM2}.
If $u$ is a precisely represented $BV$ function on $\rn$,
then $u\circ \ff$ is $BV$ on $AM$- almost every path $\ff$.
This is not true for the $M_1$ modulus.

Using the related notion of \textit{approximation upper gradient},
which is a sequence $g_j$ of positive Borel functions on $X$ such that 
$$
\operatorname{osc}_{\ff}u\le \liminf_j\int_{\ff}g_j\,ds
$$
for $AM$-almost every path,
Martio \cite{M1} introduced a version of $BV$ space on a metric measure space
$X$.

Very recently, Durand-Cartagena, Eriksson-Bique,
Korte and Shanmugalingam \cite{DCEBKS} proved that, under the assumptions
that the reference measure $\meas$ is doubling and $(X,\meas)$
supports the 1-Poincar\'e inequality, the Martio $BV$ space
coincides with the $BV$ space introduced by Miranda Jr. \cite{Mi} 
\end{rmrk}

\begin{rmrk}\label{r:reflexive}
It is clear that $AM_p\le M_p$.
In [15], we show that $AM_p=M_p$ for $1<p<\infty$ and, on the other hand,
present examples of path families $\Gamma$ with $AM(\Gamma)<M_1(\Gamma)$
\end{rmrk}

The following example shows that the modulus $M_{p,c}^A$ does depend on $\Cx(X)$.

\begin{example}\label{ex:MpneMp}
(1) Assume that $X=(0,1)$ and $\meas$ is the Lebesgue measure. Let $E\subset (0,1)$ be arbitary and
$\E=\{\delta_x\colon x\in E\}$.

Let $\Cx(X)=C_b(0,1)$ and $\Cx'(X)=C_0(0,1)$. Then it is easy to see that $M_{p,c}^A(\E)=\meas^*(E)$. In particular, if $E$ is countable, then $M_{p,c}^A(\E)=0$.
On the other hand, if  $\sup E=1$, then there is no admissible function for $\E$ in $\Cx'(X)$, hence $M_{p,c}^{A'}(\E)=+\infty$. Clearly, such $E$ may be countable.

(2) A similar example may be done using paths. It is enough to take $X=(0,1)\times[0,1]$, $\meas$ to be the Lebesgue measure and to replace Dirac measures by the length measures on $\{t\}\times[0,1]$, $t\in(0,1)$.
\end{example} 

The situation for $AM$ moduli is different. By the following theorem $AM_{p,c}^A$ does not depend on $\Cx(X)$ and, moreover, is equal to $AM_p$.

\begin{thm}\label{t:AM=AM} 
Let $p\ge 1$. Then $AM_{p}=AM_{p,c}^A$.
\end{thm}
\begin{proof} Choose $\E\subset\M^+(X)$.
Obviously $AM_p(\E)\le AM_{p,c}^A(\E)$. For the converse inequality
we may assume that $AM_p(\E)<\infty$. 
Let $(\rho_j)_j$ be an admissible sequence for $\E$ consisting of 
lower semicontinuous $L^p$ functions.

By Lemma~\ref{L:aprox}(ii) there is, for each $j\in\en$, a sequence $(u_{j,i})$ in $\Cx(X)^+$
such that $u_{j,i}\nearrow \rho_j$.
By the dominated convergence theorem we get
$$\|u_{j,i}-\rho_{j}\|_p\overset{i\to\infty}{\longrightarrow} 0.$$ 
 Choose $\ep>0$. 
 Passing to a subsequence, we may assume
that for $i,j\in\en$ we have
$$
\|u_{j,i}-\rho_j\|_p\le 2^{-j-i}\ep,
$$
so that 
$$\|u_{j,i+1}-u_{j,i}\|_p\le \|\rho_j-u_{j,i}\|_p\le 2^{-j-i}\ep,$$
hence
\eqn{m1/2}
$$
\sum_{i,j=1}^{\infty}\|u_{j,i+1}-u_{j,i}\|_p\le \ep.
$$
Denote $u_j=u_{j,1}$ and 
reorder $u_{j,i+1}-u_{j,i}$ ($i,j=1,2,\dots$) 
into a single sequence $(w_j)_j$.
Set 
$$
g_j=u_j+w_1+\dots+w_j.
$$ 
We obtain
\eqn{m1} 
$$
\|g_j\|_p\le \|u_j\|_p+\ep \le \|\rho_j\|_p+\ep
$$
from \eqref{m1/2} by the triangle inequality.

Next we show that the sequence $(g_j)$ is admissible for $\E$.
Choose $\mu\in\E$ and set
$$
t=\liminf_{j\to\infty}\int_{X}u_j\,d\mu.
$$ 
Choose $\delta>0$ and find $k\in\en$ such that 
\eqn{m2}
$$
\int_{X}u_j\,d\mu>t-\delta \text{ and } \int_{X}\rho_j\,d\mu>1-\delta,\qquad j\ge k.
$$
Find $m\ge k$ such that 
\eqn{m3}
$$
\int_{X}u_m\,d\mu<t+\delta.
$$
Then by \eqref{m2} and \eqref{m3}
\eqn{m4}
$$
\sum_{i=1}^{\infty}\int_{X}(u_{m,i+1}-u_{m,i})\,d\mu=
\int_{X}\rho_m\,d\mu-
\int_{X}u_m\,d\mu>
1-t-2\delta.
$$
Since the sequence $(w_j)_j$ contains all functions 
$u_{m,i+1}-u_{m,i}$,  there exists $l\ge k$
such that 
$$
\sum_{j=1}^l\int_{X}w_j\,d\mu >1-t-2\delta.
$$
Let $j\ge l$, then by \eqref{m2} and \eqref{m4}
$$
\int_{X}g_j\,d\mu\ge \int_{X}u_j\,d\mu + \sum_{i\le l}\int_{X}w_i\,d\mu
\ge (t-\delta)+(1-t-2\delta)=1-3\delta.
$$
We have shown that $(g_j)$ is admissible for $\E$ and by \eqref{m1},
$$
\liminf_{j}\|g_j\|_p\le \liminf_{j}\|\rho_j\|_p+\ep. 
$$
Since $g_j\in\Cx(X)$ for each $j$ and $\ep>0$ is arbitrary, we deduce  $AM_{p,c}^A(\E)\le AM_p(\E)$. 
\end{proof}

 The previous theorem says, in particular, that the modulus $AM_{p,c}^A$ does not depend on the choice of $\Cx(X)$. For $M_{p,c}^A$ the situation is different by Example~\ref{ex:MpneMp}, but we have at least the following result on compact families of measures.

\begin{thm}\label{t:eq} Let $p\ge1$ and $\K\subset \M^+(X)$ be $\taux$-compact.
Then
$$
M_{p}(\K)=M_{p,c}^A(\K).
$$
\end{thm}

\begin{proof} 
Since $M_{p,c}^A(\K) \ge M_p(\K)$, we need to prove the opposite inequality and for
that we may assume that $M_p(\K) < \infty$. 
Let $\ep\in  (0, 1/2)$. Choose an $M_p$-admissible function $\rho$ for $\K$ 
such that
$$
\int_{X}\rho^p\,d\meas< M_p(\K)+\ep.
$$
By Lemma~\ref{L:aprox}(ii) there is a sequence $(\rho_j)_j$ of functions from 
$\Cx(X)^+$ such that
$\rho_j\nearrow \rho$.
By the monotone convergence theorem we have
$$
\langle \mu,\;\rho_j\rangle \nearrow \langle \mu,\;\rho\rangle 
$$
for every $\mu\in\K$.
Next, let 
$$
\G_j=\{\mu\in\M^+(X)\colon \langle \mu,\;\rho_j\rangle>1-\ep \}.
$$
Since  $\rho_j\in\Cx(X)$, $\G_j$ form a $\taux$-open cover of $\K$.
By compactness of $\K$ and monotonicity of $(\G_j)_j$ there is
$k\in\en$ such that $\G_k\supset \K$.
Then $\rho_k/(1-\ep)$ is admissible for $M_{p,c}^A(\K)$ and we obtain
$$
M_{p,c}^A(\K)\le \int_{X}\frac{\rho_k^p}{(1-\ep)^p}\,d\meas < \frac{M_p(\K)+\ep}{(1-\ep)^p}.
$$
Letting $\ep\to 0$ we conclude that $M_{p,c}^A(\K) \le M_p(\K)$
\end{proof}

Now, we will study the moduli from the point of view of 
properties listed in Section \ref{s:choquet}.

\begin{thm}[\cite{ADS}, Theorem 5.1]\label{t:choquet} 
If $1<p<\infty$,
$M_p$ is a Choquet capacity on $(\M^+(X),\tau_b)$.
\end{thm}

\begin{rmrk} The results of \cite{ADS} concern the topology $\tau_b$.
Moreover, the notion of a Choquet capacity is defined and studied on Polish spaces, so it 
is natural to formulate this result for $\tau_b$. However, even for nonmetrizable topologies
it has a sense to study the properties from Definition \ref{d:capacity}.
Observe that properties \eqref{monot} and \eqref{upward} are independent
of the choice of topology. Concerning \eqref{downward}, it is shown for 
$\tau_b$ in \cite{ADS}, but for more general $\taux$ it holds as well, see Proposition~\ref{L:wupper} below.
\end{rmrk}

By Theorem \ref{t:choquet}, the $M_p$-modulus, $p>1$, has the property \eqref{upward}
of Definition \ref{d:capacity}, namely
\begin{equation} \label{E1}
\lim_{i \rightarrow \infty }M_p(\E_i) = M_p\Bigl(\bigcup_i \E_i\Bigr)
\end{equation}
if $\E_i\subset\M^+(X)$,
$\E_1 \subset \E_2 \subset \, \dots$.

The situation for $p=1$ is more complicated. The property \eqref{upward} 
of Definition \ref{d:capacity} holds
for neither of the moduli $AM$, $M_1$ and $M_{1,c}^A$, so that 
these moduli fail to be Choquet capacities. This is shown in Theorems
\ref{t:nonincr-meas}, \ref{t:m1nec} and \ref{t:cnec} below. 

To describe the situation in detail, we start with some positive results.

\begin{lemma} \label{Lbasic}
If $\E_1 \subset \E_2 \subset \, ...$ is a sequence of subsets of $\M^+(X)$
and $\E = \bigcup_i \E_i$, then 
$$ 
AM(\E) \leq \lim_{i \rightarrow \infty} M_1(\E_i) .
$$
\end{lemma}

\begin{proof}
We may assume
that $ \lim\limits_{i \rightarrow \infty} M_1(\E_i) < \infty$. Choose
admissible functions $\rho_i$ for $\E_i$ such that
$$\int_X \rho_i \, d \meas < M_1(\E_i) + 2^{-i},\qquad i=1,2,\dots.$$
The sequence $(\rho_i)$ is admissible for $\E$. Indeed,
if $\mu \in \E$, then $$\mu \in \E_k\subset \E_{k+1}\subset\dots$$ for
some $k$ and thus
$$ \liminf_{i \rightarrow \infty} \int_{X} \rho_i \, d\mu \geq 1 .$$
Now we obtain
$$ AM(\E) \leq  \liminf_{i \rightarrow \infty} \int_{X} \rho_i \, d \meas
\leq  \liminf_{i \rightarrow \infty} (M_1(\E_i) + 2^{-i}) = 
\lim_{i \rightarrow \infty} M_1(\E_i).$$
\end{proof}

\begin{cor} \label{C2}
If $\E_1 \subset \E_2 \subset \, ...$ are subsets of $\M^+(X)$
and
$M_1(\E_i) = AM(\E_i)$ for each $i$, then 
\begin{equation} \label{E3}
\lim_{i \rightarrow \infty}AM(\E_i) = 
AM\Bigl(\bigcup_i \E_i\Bigr) 
\end{equation}
\end{cor}

\begin{proof} 
By Lemma \ref{Lbasic}, 
$$AM\Bigl(\bigcup_i \E_i\Bigr)  \le
\lim\limits_{i \rightarrow \infty}M(\E_i)=\lim\limits_{i \rightarrow \infty}AM(\E_i) .$$ 
The reverse inequality follows from monotonicity.
\end{proof}

The following theorem provides an alternative characterization for the $AM$-modulus in terms of increasing path families and the $M_1$--modulus.

\begin{thm} \label{Tinside}
If $\E\subset \M^+(X)$, then 
\begin{equation} \label{Eepsilon} 
\aligned
AM(\E) &= \inf \{ \lim_{i \rightarrow \infty} M_1(\E_i) \colon \, \E_1 \subset \E_2 \subset \, ... \textnormal{ and } \bigcup_i \E_i = \E \}
\\&= \inf \{ \lim_{i \rightarrow \infty} M_{1,c}^A(\E_i) \colon \, \E_1 \subset \E_2 \subset \, ... \textnormal{ and } \bigcup_i \E_i = \E \}.
\endaligned
\end{equation}
\end{thm}

\begin{proof} 
 Fix $\E\subset\M^+(X)$. Since clearly $M_1\le M_{1,c}^A$, by Lemma \ref{Lbasic} it suffices to show 
that for each $\varepsilon > 0$ there is an increasing  sequence $(\E_i)$ as 
in
\eqref{Eepsilon}
and
$$ \lim_{i \rightarrow \infty} M_{1,c}^A(\E_i) \leq AM(\E)  + \varepsilon.$$
Assume first that $AM(\E) < \infty$. 
Fix $\delta \in (0,1)$ to be specified later
 and choose an admissible
sequence $(\rho_j)$ for $\E$ such that
$$ \liminf_{j \rightarrow \infty} \int_{X} \rho_j \, d \meas \leq AM(\E) + \delta.$$
In view of Theorem \ref{t:AM=AM}, we 
may assume that
 $\rho_j$ in $\Cx(X)$.
Let
$$ \E_i = \bigcap_{j\ge i}\Bigl\{ \mu \in \E: \, \int_{X} \rho_j \, d\mu \geq 1-\delta  \Bigr\}.$$
Then $\E_i \subset \E_{i+1}$ and $\E = \bigcup_i \E_i$. Since
$\rho_i/(1 - \delta)$ is admissible for $\E_i$, we obtain
$$\lim_{i \rightarrow \infty}M_{1,c}^A(\E_i)\leq \liminf_{i \rightarrow \infty} \int_{X} \frac{\rho_i}{1-\delta} \, d \meas
\leq \frac{AM(\E) +\delta}{1-\delta}<AM(\E) + \varepsilon$$
if $\delta$ is small enough.

If $AM(\E) = \infty$, then we can choose  $\E_i = \E$ for each $i$.
\end{proof}

\begin{thm} \label{TMAM}
Let $\E\subset\M^+(X)$.
Then the following assertions are equivalent:
\begin{enumerate}[\rm(i)]
\item
$M_1(\E)=AM(\E)$,
\item\label{E2}
 $M_1(\E)=\lim\limits_{i \rightarrow \infty}M_1(\E_i)$ 
 for each increasing sequence $(\E_i)$ of subsets of
  $\M^+(X)$
with $\E = \bigcup_i \E_i$. 
\end{enumerate}
\end{thm}

\begin{proof} (i)$\implies$(ii):
If $AM(\E) = M_1(\E)$ then (\ref{E2}) follows from Lemma \ref{Lbasic} as
$$
\lim_{i \rightarrow \infty} M_1(\E_i) \leq M_1(\E).
$$ 
(ii)$\implies$(i):
Since $AM(\E)\le M_1(\E)$, we need only to prove 
the converse inequality and we may assume that $AM(\E) < \infty$. 
Choose $\ep>0$ and use Lemma \ref{Tinside} to find an increasing sequence
$(\E_i)$ of subsets of $\E$ with $\E = \bigcup_i \E_i$ such that
$$
\lim_i M_1(\E_j) < AM(\E)+\ep.
$$
Then by (ii)
$$
M_1(\E)<AM(\E)+\ep
$$
and letting $\ep\to0$ we obtain the desired inequality.
\end{proof}

We give two examples involving $\tau_b$-compact families. Since $\tau_b$ is the finest
from all $\taux$-topologies, 
the families in consideration are $\taux$-compact as well.

\begin{thm}\label{t:cnec}
Let $X$ be $\er^2$ with the Lebesgue measure.
There exists an increasing
sequence $\Gamma_k$ of $\tau_b$-compact families of paths in $\er^2$
such that, denoting $\Gamma=\bigcup_k\Gamma_k$,
$$
M_1(\Gamma)=\lim_k M_{1}(\Gamma_k)=\lim_k M_{1,c}^A(\Gamma_k)=0<\infty= M_{1,c}^A(\Gamma).
$$
Consequently, $\Gamma$ is $\F_{\sigma}$ but not $M_{1,c}^A$-capacitable
(see Definition \ref{d:capacitable}).
\end{thm}

\begin{proof} For any $k\in\en$, let $\Gamma_k$ be the 
family 
of all paths
$\ff_{z}(t)=tz$, $t\in [0,1]$, 
where $z\in\overline B(0,1)\setminus B(0,1/k)$.
Then $\Gamma_k$ are $\tau_b$-compact.
If $\rho\colon \er^2\to \er$ is continuous and $z_k=(1/k,0)$,
then
$$
\inf_{k}\int_{\ff_{z_k}}\rho\,ds= 0,
$$
so that there is no continuous admissible function for $\Gamma$ and 
$M_{1,c}^A(\Gamma)=\infty$.
On the other hand, for any $\ep>0$, $\ep|x|^{-1}$ is an admissible function
for $\Gamma$, so that $M_1(\Gamma)=0$.
Let $\omega$ be a nonnegative continuous function on $[0,\infty)$ with a support in $[0,1]$
such that $\int_{0}^{\infty}\omega(t)\,dt=1$. 
Write
$$
\rho_j(x)=j\omega(j|x|).
$$
Then, for fixed $k$,
$j\ge k$, and $z\in\overline B(0,1)\setminus B(0,1/k)$, we have
$$
\int_{\ff_z}\rho_j\,ds\ge \int_0^{1/j}j\omega(jr)\,dr=1.
$$
Thus
$\rho_j$, $j\ge k$, are admissible functions for $\Gamma_k$
and
$$
\inf_{j\ge k}\int_{\er^2}\rho_j(x)\,dx=
\inf_{j\ge k}\int_0^{1/j}2\pi r\,j\omega(jr)\,dr=0.
$$
Therefore, $M_{1,c}^A(\Gamma_k)=0$, $k=1,2,\dots$.
\end{proof}

The following theorem is based on a deeper study of an example from \cite{M1}.

\begin{thm}\label{t:m1nec}
There exists an increasing
sequence $\Gamma_k$ of $\tau_b$-compact families of paths in $\er$
such that, denoting $\Gamma=\bigcup_k\Gamma_k$,
$$
AM(\Gamma)\le\lim_{k\to\infty} M_1(\Gamma_k)\le 1,
$$
but
$$
M_1(\Gamma)=\infty.
$$
Thus, $AM(\Gamma)< M_1(\Gamma)$ and $\Gamma$ is $\F_{\sigma}$, but 
not $M_1$-capacitable.
\end{thm}

\begin{proof} 
Let
$\Gamma_k$  be the 
family 
of all paths $\ff_{r}(t)=t$, $t\in [0,r]$, where $r\in[ 2^{-k},1]$.
Then $\Gamma_k$ are $\tau_b$-compact.
If $\rho\in L^1(\er)$, then
$$
\inf_{r>0}\int_{\ff_{r}}\rho(t)\,dt= 0,
$$
so that there is no admissible function for $\Gamma$ and 
$M_{1}(\Gamma)=\infty$.
On the other hand, if $\eta\ge 0$ is a continuous function on $[0,\infty)$
with  support on  $[0,\frac12]$ such that $\int_{0}^{\infty}\eta(x)\,dx=1$,
and
$$
\rho_k(x)= 2^{k}\eta(2^{k}x),
$$
then $\rho_k$ is admissible for $\Gamma_k$, $k=1,2,\dots$.
It follows
$$
\liminf_{k\to\infty} M_1(\Gamma_k)\le 1.
$$
By Lemma \ref{Lbasic},
$$
AM(\Gamma)\le\liminf_{k\to\infty} M_1(\Gamma_k)  \le 1. 
$$
\end{proof}

\begin{rmrk} Let $\E_1 \subset \E_2 \subset \, ...$ be a sequence of subsets of $\M^+(X)$ and $\E=\bigcup_j\E_j$. 
Then, in general, 
there is no relation between $M_1(\E)$ and 
$\lim_j M_{1,c}^A(\E_j)$. Indeed, we have
$$
M_1(\E)<
\lim_j M_{1,c}^A(\E_j)
$$
if $\E_j=\E$ and $M_1(\E)<M_{1,c}^A(\E)$ as in Theorem \ref{t:cnec}.
On the other hand, we have
$$
\lim_j M_{1,c}^A(\E_j)< M_1(\E)
$$
if $\E$ is a set with $AM(\E)< M_1(\E)$ 
as in Theorem \ref{t:m1nec}
and $\E_j$ are as in Theorem
\ref{Tinside}.
\end{rmrk}

Next we focus on the failure of property  \eqref{upward}
of Definition~\ref{d:capacity}. We first present a general construction of a counterexample and next we apply it to show that this property fails in most cases for families of measures and at least in some cases also for paths families.

\begin{lemma}\label{l:construction}
Let $X$ be a Polish space equipped with a reference Radon measure $\meas$.
Let $(G_{m,i})_{m,i\in\en}$ be a system of open subsets of $X$ with following properties:
\begin{enumerate}[\rm(a)]
\item $G_{m,1}$, $m\in\en$ are pairwise disjoint.
\item $G_{m,1}\supset G_{m,2}\supset G_{m,3}\supset\dots$ for each $m\in\en$.
\item $\meas (G_{m,i})>0$ for each $m,i\in\en$.
\item $\lim_{i\to\infty}\meas(G_{m,i})=0$ for  each $m\in\en$.
\item $\meas(\bigcup_m G_{m,1})<\infty$.
\end{enumerate}
For each $m\in \en$ and each sequence $\boldsymbol s=(s_n)_{n}$
of integers we consider the set 
$$
H_{m,\boldsymbol s}=\bigcup_{n\ge m}G_{n,s_n}
$$
and denote the restriction of $\meas$ to $H_{m,\boldsymbol s}$ by $\mu_{m,\boldsymbol s}$.
Let $\E_0\subset\M^+(X)$ contain all measures $\mu_{m,\boldsymbol s}$. 
Then there is an increasing sequence $(\E_m)_{m\ge 1}$ of subsets of $\E_0$ that 
$$
AM\Big(\bigcup_m \E_{m}\Big)>\lim_m AM(\E_m).
$$
\end{lemma}

\begin{proof}
Set 
$$g_{m,i}=\frac{1}{\meas(G_{m,i})}\chi_{G_{m,i}},\qquad m,i\in\en.$$
Then each $g_{m,i}$ is a lower semicontinuous function satisfying $\|g_{m,i}\|_1=1$.
Set
$$\E_m=\bigcap_{n\ge m}\left\{\mu\in\E_0\colon \liminf_i \int_X g_{n,i}\,d\mu\ge 1\right\}.$$
By the very definition the sequence 
$(g_{m,i})_i$ 
is admissible for $\E_m$, hence $AM(\E_m)\le1$.
Clearly $\E_{1}\subset \E_2\subset \E_{3}\subset\dots$.
Set $\E=\bigcup\limits_{m}\E_m$.
We are going to show that $AM(\E)\ge2$.

Assume that $AM(\E)<2$. Then there is a sequence $(h_k)_k$ admissible for $\E$ such that $\liminf \|h_k\|_1<2$. 
Since a subsequence of an admissible sequence is again admissible we may assume that there is some $\varepsilon>0$ such that $\|h_k\|_1<2(1-\varepsilon)$ for each $k\in\en$.
For each $m\in\en$ let $\mu_m$ be the restriction of $\meas$ to
$\bigcup_{n \ge m} G_{n,1}$. 
Then $\mu_m=\mu_{m,\boldsymbol s}$ for $\boldsymbol s=(1,1,1,\dots)$, therefore
$\mu_m\in \E_0$. 
We have
$$\int_X g_{n,i} \,d\mu_m=1
\mbox{ for }n\ge m,
\ i\in\en,$$
thus $\mu_m\in \E_{n} \subset \E$ 
for $n\ge m$.
It follows that $\liminf_k \int_X h_k\,d\mu_m\ge 1$. Thus there is
$p_m\in\en$ such that for each $k\ge p_m$ we have
$$1-\varepsilon<\int_X h_k\,d\mu_m=\int_{\bigcup\limits_{n\ge m}G_{n,1}} h_k\,d\meas.$$
It follows that for $k\ge p_m$ we have
$$\int_{\bigcup\limits_{n< m}G_{n,1}} h_k\,d\meas
\le \|h_k\|_1-\int_{\bigcup\limits_{n\ge m}G_{n,1}} h_k\,d\meas<1-\varepsilon.$$
%Note that without loss of generality the sequence $(p_m)_m$ is increasing.

Now, using (d) we find an increasing sequence $\boldsymbol q=(q_m)_m$ of  positive integers such that
$q_m\ge p_m$ and 
$$\int_{G_{m,q_m}} h_k\,d\meas<\frac{\ep}{2^{i+1}}\mbox{ for }k\le q_m.
$$
Let $\nu$ be the the restriction of $\meas$ to $\bigcup_m G_{m,q_m}$. 
Then $\nu=\mu_{1,\boldsymbol q}\in \E_0$. 
Moreover, $\nu\in \E_1\subset\E$ as
$$\int_X g_{m,i}\, d\nu=1,\qquad
%\mbox{ whenever }
i\ge q_m.$$
Then  for each $k\in\en$ we have
$$\begin{aligned}
\int_X h_k\,d\nu&=\int_{\bigcup_m G_{m,q_m}} h_k\,d\meas
\\&\le \int_{\bigcup\limits_{\{m\colon q_m<k\}} G_{m,q_m}} h_k\,d\meas +
\int_{\bigcup\limits_{\{m\colon q_m\ge k\}} G_{m,q_m}} h_k\,d\meas
\\&<1-\ep+\sum_{i=1}^\infty\frac{\ep}{2^{i+1}}=1-\frac\ep2.
\end{aligned}
$$
Therefore $(h_k)_k$ cannot be admissible for $\E$, which is a contradiction
completing the proof.
\end{proof}

\begin{rmrk}\label{r:borel} 
It is clear that the families $\E_m$ constructed in the proof of the above lemma are Borel subsets of 
$\E_0$.
\end{rmrk}

In the following theorem we show that the modulus $AM$  satisfies property \eqref{upward} from Definition~\ref{d:capacity} is fulfilled only in trivial cases. We formulate the result as an equivalence to provide a complete picture, but of course, the `only if part' is more important
for the theory.

\begin{thm}\label{t:nonincr-meas}
Let $X$ be a Polish space equipped with a reference Radon measure $\meas$. Then the modulus $AM$  satisfies property \eqref{upward} from Definition~\ref{d:capacity} in $\M^+(X)$ if and only if 
$$\meas=\sum_{x\in F} c_x\delta_{x},$$
where $F$ is a closed discrete subset of $X$ (finite or infinite) and $(c_x)_{x\in F}$ are positive numbers bounded below by some $c>0$.
\end{thm}

\begin{proof}
Let us start by the `if part'. Assume that $\meas$ has the required form. First observe, that for each $\E\subset\M^+(X)$ we have
$$AM(\E)=AM(\{\mu\in \E\colon \mu(X\setminus F)=0\}).$$
Indeed, the inequality $\ge$ is obvious. To prove the converse one fix any sequence $(h_k)$ admissible for $\{\mu\in \E\colon \mu(X\setminus F)=0\}$ and observe that the sequence $(h_k+k\chi_{X\setminus F})$ is 
admissible
for $\E$ and $\|h_k+k\chi_{X\setminus F}\|_1=\|h_k\|_1$ for each $k\in\en$. Hence the inequality $\le$ follows as well.

Hence, to complete the proof of the `if part' we may assume that $F=X$. We give the proof in case $X$ is infinite. The proof in case $X$ is finite is analogous (and easier). Hence we may assume $F=X=\en$. Then $\M(X)$ is canonically identified with the Banach space $\ell_1$
through the mapping $\varkappa\colon \M(X)\to\ell_1$ defined by $\varkappa(\mu)(k)=\mu(\{k\})$.
(We consider $\ell_1$ rather as a space of functions than a space of sequences as the comparison
with $L^1(\meas)$ is then more accurate.)
 We will show that $AM(\E)=M_1(\E)$ for any $\E\subset\M^+(X)$. Since the inequality $\le$ holds always, it is enough to  prove the converse one. If $AM(\E)=\infty$, the inequality is obvious. So, assume that $AM(\E)<\infty$ and fix any $\ep>0$. Then there is 
 an admissible sequence $(\rho_n)_n$ for $\E$ 
such that $\liminf\|\rho_n\|_1<AM(\E)+\ep$. 
 Up to passing to a subsequence we may assume that $\|\rho_n\|_1<AM(\E)+\ep$ for each $n\in\en$. 
Since $\|f\|_1\ge c\|f\|_{\ell_1}$ for each $f\in L^1(\meas)$, 
we deduce that the sequence $(\rho_n)$ is bounded in $\ell_1$.
Therefore, up to passing to a subsequence we may assume that the sequence $(\rho_n)_n$ 
converges to some $\rho\in\ell_1$
in the weak* topology of $c_0^*$.
In particular, $\rho_n\to \rho$ pointwise, hence $\|\rho\|_1\le AM(\E)+\ep$ by the Fatou lemma. Further, 
$\varkappa(\mu)\in \ell^1\subset c_0$ for 
any $\mu\in \E$, thus
$$\int_{\en} \rho\,d\mu=\langle \rho,\varkappa(\mu)\rangle=\lim_n \langle \rho_n,\varkappa(\mu)\rangle=
\lim_n \int_{\en} \rho_n\,d\mu\ge1,$$
so that 
$\rho$ is admissible for $\E$. It follows that $M_1(\E)\le AM(\E)+\ep$. Since $\ep>0$ is arbitrary, we get $M_1(\E)\le AM(\E)$, thus $M_1(\E)=AM(\E)$. 

Having proved $M_1=AM$, the validity of property \eqref{upward} of Definition~\ref{d:capacity} follows from Corollary~\ref{C2}.
\smallskip

We continue by proving the `only if part'. Assume that $\meas$ is not of the given form. Then there are two possibilities:

Case 1: The support of $\meas$ is not discrete. It follows that there is a one-to-one sequence $(x_n)$ in the support of $\meas$ converging to some $x\in X$. Since $\meas$ is finite on the compact set $\{x\}\cup\{x_n\colon n\in\en\}$, necessarily $\sum_n \meas(\{x_n\})<\infty$. Now, using outer regularity of $\meas$ it is easy to find a disjoint sequence $(U_n)$ of open sets such that $\meas(U_n)>0$ for each $n\in\en$ and $\sum_n\meas(U_n)<\infty$.

Case 2: The measure $\meas$ has the above form with $c_x>0$ but $\inf_x c_x=0$. Then one can choose a sequence $(x_n)$ in the support with $\sum_n c_{x_n}<\infty$. Again, using outer regularity of $\meas$ it is easy to find a disjoint sequence $(U_n)$ of open sets such that $\meas(U_n)>0$ for each $n\in\en$ and $\sum_n\meas(U_n)<\infty$.

From the sequence $(U_n)$ constructed in both cases one may easily construct a family $(G_{m,i})$ with properties (a)--(e) from Lemma~\ref{l:construction}. 
For example, if $(p_m)_m$ is a sequence of distinct primes, we may set
$$
G_{m,i}=\bigcup_{j\ge i}U_{p_m^j}.
$$
This completes the proof.
\end{proof}

Lemma~\ref{l:construction} may be also used to provide a counterexample for families of paths given in the following statement.

\begin{example}\label{ex:nonincr-paths}
There is a compact metric space $X$ equipped with a (doubling) reference Radon measure $\meas$ such that the modulus $AM$ fails property  \eqref{upward} of Definition~\ref{d:capacity} on families of paths.
\end{example}

\begin{proof}
The space $X$ will be constructed as a suitable subset of $\er^2$. 
Denote 
$$
L_m= \{x=(x_1,x_2)\in  \er^2: \, x_1 \in [0,2^{-m}],\ x_2=x_1/m  \},
$$
a collection of line segments emerging from $0$, and set 
$$
X=
\bigcup_{m=1}^{\infty}L_m.
$$
It is clear that $X$ is a compact subset of $\er^2$.
Let $\meas$ be the double of the linear measure. 
(The linear measure itself would serve as well, but with its double we can use
Lemma \ref{l:construction} directly.)
Note that 
$\meas(X) < \infty$ and  $\meas$ is a doubling
measure in $X$, 
see Remark \ref{r:doubling}.
Further, for $m,i\in\en$ set
$$G_{m,i}=\{x\in L_m\colon 0<x_1<2^{-m-i+1}\}.$$
These sets are open subsets of $X$ and, moreover, they satisfy conditions (a)--(e) from Lemma~\ref{l:construction}. Moreover, each of the measures $\mu_{m,\boldsymbol s}$ described in the lemma is provided by a path (we may take a path
 which runs twice over 
 $\{x\in L_n\colon x_1\le 2^{-n-s_n+1}\}$ 
 for each 
 $n\ge m$). 
Thus we may conclude by using Lemma~\ref{l:construction}.
 \end{proof}

\begin{rmrk}\label{r:doubling}
It is easy to observe that the measure $\meas$ from Example \label{ex:nonincr-paths} is doubling,
which is interesting for experts in analysis on metric measure spaces.
It is obvious that $\meas(B(0,2r))\le 2\meas B(0,r)$ for any $r>0$ and that $\meas(B(x,2r))\le 2\meas B(x,r)$
if $x\in X$ and $r$ is sufficiently small. The function
$$
f(x,r)=
\begin{cases}
 \frac{\meas(B(x,2r))}{\meas(B(x,r))},& r> 0,\\
 2,& r=0
 \end{cases}
$$
is upper semicontinuous on the compact space $X\times [0,2]$ and thus it attains a maximum.
\end{rmrk}

\begin{rmrk}
Let us remark that while failure of property \eqref{upward} from Definition~\ref{d:capacity} for families of measures is a rather common feature by Theorem~\ref{t:nonincr-meas}, the counterexample for families of paths is rather special. The reason is the use of Lemma~\ref{l:construction} where suitable restrictions of the reference measure play a key role. So, the example is designed in such a way that these
restrictions are measures generated by paths. 
This approach is impossible 
in $\rn$.
So, the question whether property \eqref{upward} from Definition~\ref{d:capacity} 
holds for families of paths in 
$\rn$ remains open.
\end{rmrk}

\begin{thm}\label{t:nonouter}
$AM$ is not outer regular on $(\M^+(X),\taux)$.
\end{thm} 

\begin{proof}
Since $\tau_b$ is finer than $\taux$, it is enough to consider $\tau_b$-openness.
Let $X=\er$ and $\meas$ be the Lebesgue measure. 
We identify any closed 
intervals $I\subset\er$ with the one-to one path with locus $I$
(or with the restriction of the Lebesgue measure on $I$).
Let $\Gamma$ be the family of all paths $[0,\ell]$, $\ell\in (0,1]$.
Then $AM(\Gamma)= 1$. Indeed, the inequality $AM(\Gamma)\ge 1$ is 
obvious and irrelevant. For the converse we use
the admissible sequence $(\rho_k)_k$, where
$$
\rho_k(x)=
\begin{cases}
k,& x\in (0,\frac1k),
\\
0&\text{otherwise.} 
\end{cases}
$$
Choose an $\tau_b$-open set 
$\G\subset\M^+(\er)$ 
such that $\Gamma\subset\G$.
Then there exists $\delta_1>0$ such that $[\delta_1,1]$ belongs to $\G$. 
Let $(g_k)_k$ be an admissible sequence for $\G$.
Then 
$$
\liminf_{k}\int_{\delta_1}^1g_k\,dx\ge 1,
$$
but since $[0,\delta_1]\in\Gamma$, we have also
$$
\liminf_{k}\int_0^{\delta_1}g_k\,dx\ge 1.
$$
Altogether,
$$
\liminf_{k}\int_0^1g_k\,dx\ge 2,
$$
which shows that $AM(\G)\ge2$.

This is enough to disprove the outer regularity. However, we can proceed further
to get that in fact $AM(\G)=\infty$.
Indeed, by induction we find $\delta_j$, $j=1,2,\dots$, such that $[\delta_{j+1},\delta_j]\in\G$
for each $j\in\en$.

\end{proof}

\begin{rmrk} 
We can construct a similar example on paths in $\er^2$ equipped with the 
Lebesgue measure. For $x,a,b\in\er$ with $a<b$ let $\gamma_{x,a,b}$ the be path defined
by
$$t\mapsto (x,t),\quad t\in[a,b],$$
which we identify with the respective length measure. It easily follows from the description of
the topology $\tau_b$ in Lemma~\ref{L:taub}(i) that the assignment $(x,a,b)\mapsto \gamma_{x,a,b}$ is continuous 
as a mapping to $(\M^+(\er^2),\tau_b)$.
Let 
$$\Gamma=\{\gamma_{x,0,\ell}\colon x\in(0,1),\ell\in(0,1]\}.$$
 Again $AM(\Gamma)\le1$ as witnessed by the admissible sequence $(\rho_k)$ where
$\rho_k=k\chi_{(0,1)\times(0,\frac1k)}$.
 Let $\G\supset \Gamma$ be $\tau_b$-open.
	
For each $m\in\en$ let
$$E_m=\bigcap_{j\ge m}\{x\in (0,1)\colon \gamma_{x,\frac1j,1}\in \G\}.$$
Then each $E_m$ is a $\G_\delta$-subset of $(0,1)$. Moreover, since for each $x\in (0,1)$ we have
%$$\gamma_{x,\frac1n,1}\mathrel{\mathop{\longrightarrow}_{\tau_b}^{n\to\infty}} \gamma_{x,0,1},$$
$$
\gamma_{x,0,1}=\tau_b\text{-}\!\!\lim_{n\to\infty}\gamma_{x,\frac1n,1}\,,
$$
we deduce that $\bigcup_m E_m=(0,1)$. 
	
Hence there is $m\in\en$ with $|E_m|>\frac34$. Then for each admissible sequence $(g_k)_k$ for $\G$ we have
$$\begin{aligned}
\liminf_{k}\int_{E_m\times [0,1]}&g_k(x,y) \, dx\, dy &=  
\liminf_{k}\int_{E_m} \Bigl(\int_0^1 g_k(x,y) \, dy\Bigr)\, dx 
\\&
\ge \int_{E_m} \left(\liminf_{k} \int_0^1 g_k(x,y) \, dy\right)\, dx 
\\& \ge \int_{E_m} \left(\liminf_{k} \int_0^{1/m} g_k(x,y) \, dy+\liminf_{k} \int_{1/m}^{1} g_k(x,y) \, dy\right)\, dx \\&\ge 
2|E_m|
\ge\frac32.
\end{aligned}
$$
Hence $AM(\G)\ge\frac32>1$.
\end{rmrk}

\begin{rmrk}\label{r:wupper} The moduli $AM_p$, $M_p$ and $M_{p,c}^A$
nevertheless satisfy \eqref{wupper} and thus 
\eqref{downward} of Definition \ref{d:capacity} for each $p\ge 1$ (see Proposition~\ref{L:wupper} below).
\end{rmrk}

By Theorem~\ref{t:nonouter}, the modulus $AM$ is not outer regular. However, a weaker version of outer regularity holds by the following lemma.

\begin{lemma}\label{l:efsigma} 
Let $\E\subset\M^+(X)$  be arbitrary.
Then 
$$AM(\E)=\inf\{AM(\A)\colon \A\supset \E\text{ is }\taux\mbox{-}\F_{\sigma}\}$$
\end{lemma}

\begin{proof}
The inequality $\le$ is obvious. Let us prove the converse one.
Choose $0<t<1$.
Let $(\rho_k)$ be an admissible sequence for $\E$ made of functions from $\Cx(X)^+$ such that 
$$
\liminf_k\|\rho_k\|_1<AM(\E)+1-t.
$$
(We may use test functions from $\Cx(X)^+$ by Theorem \ref{t:AM=AM}.)
Set 
$$
\F_k=\{\mu\in  \M^+(X)  \colon
\langle \mu,\rho_j\rangle
\ge t\mbox{ for }j\ge k\}.
$$
Then $\F_k$ are $\taux$-closed and
$$
\E\subset \A:=\bigcup_k\F_k.
$$
Further, the sequence $(\rho_k/t)_k$ is admissible for $\A$.
Therefore
$$
AM(\A)\le \frac1t (AM(\E)+1-t) 
$$
and letting $t \rightarrow 1$ we obtain
$$
\inf\{AM(\A)\colon \A\supset \E\text{ is }\taux\mbox{-}\F_{\sigma}\}\le AM(\E),
$$
which completes the proof.
\end{proof}

\section{Plans and a content on $\M^+(X)$}\label{sec:plans}

In this section $X$ continues to be a Polish space equipped with a positive Radon measure $\meas$.
$\Cx(X)$ is again a fixed acceptable subspace of $C_b(X)$.

The space $\M^+(\M^+(X))$ is defined as the family of all nonnegative finite
Borel measures on $\M^+(X)$. Here we do not need to distinguish between $\tau_b$-Borel subset and $\taux$ Borel subsets of $\M^+(X)$ as they are the same, see Proposition~\ref{P:tauA}(h). Moreover, all measures from $\M^+(\M^+(X))$ are Radon on $(\M^+(X),\taux)$. Indeed, for $\tau_b$ this follows from Proposition~\ref{P:tauA}(a) and Proposition~\ref{p:outer}. Since $\taux$ is a weaker topology, it has more compact sets, so the Radon property remains true for $\taux$.

\begin{definition}[\cite{AGS,ADS}] 
A finite positive Borel measure
$\bnu$ on $\M^+(X)$ is called a \textit{plan}.
We denote
$$
\bnu^{\#}\colon E\mapsto \int_{\M^+(X)}\mu(E)\,d\bnu(\mu),\quad E\subset X\text{ Borel}.
$$
If $\bnu$ is a probability measure, $\bnu^\#$ has the interpretation as \textit{barycenter},
cf.\ \cite{ADS}.

If $\bnu^{\#}$ is absolutely continuous with respect to $\meas$, we identify it 
with its density. Under this convention, we can associate the $L^q(\meas)$ norm
$\|\bnu^\#\|_q$ with the measure $\bnu^\#$.

Now, if $p\in [1,\infty)$, 
$q$ is the dual exponent to $p$,
$\E$ is a subfamily of $\M^+(X)$,
we define the \textit{$p$-content} of $\E$ as 
\eqn{cpl}
$$
\C_p(\E)=\sup\Bigl\{\bnu^*(\E)\colon 
\bnu\in\M^+(\M^+(X)),\ \bnu^\#\ll\meas,\ 
\|\bnu^{\#}\|_q\le 1
\Big\}.
$$
Recall that $\bnu^*$ is the outer measure induced by
$\bnu$, see Proposition \ref{p:outer}, and 
the sign $\ll$ denotes the relation of absolute continuity.
We simplify $\C_1$ to $\C$ for $p=1$.
\end{definition}

\begin{rmrk}
In \cite{ADS}, this content is defined for universally measurable sets $\E$.
Our novelty is that we use the outer measure, which allows to define the content 
for all sets. Surprisingly (and in contrast with properties of $AM$ and $M_1$), this set
function is a Choquet capacity, see Theorem \ref{t:choquetness}. The following theorem
is a key step to this observation. 
\end{rmrk}

\begin{thm}\label{t:incr} Let $\E_j\subset\M^+(X)$, 
$\E_1\subset \E_2\subset\dots$, and $\E=\bigcup_j\E_j$.
Then
$$
 \C_p(\E)= \lim_j\C_p(\E_j).
$$
\end{thm}

\begin{proof}  The inequality $\ge$ is obvious. To prove the converse one
we may assume that $\C_p(\E)>0$.
Choose $0\le t<\C_p(\E)$. Then we can find 
a measure $\bpi\in \M^+(\M^+(X))$  such that  $\bpi^\#\ll\meas$,
$\|\bpi^\#\|_q\le 1$ and $\bpi^*(\E)>t$.
Since $\bpi^*$ is a Choquet capacity (Proposition \ref{p:outer}),
there is $k\in\en$ such that $\bpi^*(\E_k)>t$.
Hence $\C_p(\E_j)>t$ for all $j\ge k$.
\end{proof}

\section{Modulus and content}\label{sec:m-c}

The following fundamental theorem 
is due to Ambrosio, Di Marino and Savar\'e.

\begin{thm}[\cite{ADS}]\label{t:their} 
Let $p\in (1,\infty)$ and
$\E\subset \M^+(X)$ be a $\tau_b$-Suslin set. Then
$$
\C_p(\E)=M_p(\E)^{1/p}.
$$
\end{thm}

Since we are deeply interested in $M_1$ and $AM$-moduli corresponding to $p=1$,
we were motivated to look what happens for $p=1$. 

\begin{prop}\label{p:cplem}
Let $\E \subset \M^+(X)$ be a Borel set.
Then
$$
\C (\E)\le AM(\E). %\le M_1(\E)\le M_{1,c}^A(\E).
$$
\end{prop}

\begin{proof} 
Let $\bnu$ be a plan with  $\bnu^\#\ll\meas$,  $\|\bnu^\#\|_{\infty}\le1$ and $(\rho_j)_j$ be a 
sequence admissible for $\E$.
%Assume that $\spt\bnu\subset\E$.
Then by the Fatou lemma,
$$
\aligned
\bnu(\E)&\le \int_{\E}\Bigl(\liminf_j\int_{X}\rho_j\,d\mu\Bigr)\,d\bnu(\mu)\le
\liminf_j\int_{\E}\Bigl(\int_{X}\rho_j\,d\mu\Bigr)\,d\bnu(\mu)
\\&=\liminf_j\int_{X}\rho_j\,d\bnu^{\#}\le \|\bnu^{\#}\|_{\infty}
\liminf_j\int_{X}\rho_j\,d\meas  \le \liminf_j\int_{X}\rho_j\,d\meas. 
\endaligned
$$
Passing to the supremum on the left and to the infimum on the right we obtain the 
desired inequality $\C (\E)\le AM(\E)$.
\end{proof}

\begin{thm}\label{t:compok}
Let 
$\K\subset \M^+(X)$ be a $\taux$-compact family of measures.
Then
$$
\C (\K)= AM(\K)=M_1(\K)= M_{1,c}^A(\K).
$$
\end{thm}

\begin{proof}  This can be shown as in the first step of \cite[Theorem 5.1]{ADS}.
We simplify a bit the argument,  on the other hand, some
parts of the proof are longer in order to cover the axiomatic approach.

Since inequalities $AM(\K)\le M_1(\K)\le M_{1,c}^A(\K)$ are obvious and 
in view Proposition \ref{p:cplem}, it remains to verify that
$M_{1,c}^A(\K)\le \C(\K)$.

If $0\in\K$ then clearly $\C(\K)=\infty$, hence we 
may assume that $0\notin\K$. Also we may assume that $\xi:= M_{1,c}^A(\K)>0$, in particular,
$\K$ is nonempty. 

Consider the bounded linear operator $J\colon \Cx(X)\to C(\K)$ 
defined as 
$$Jg(\mu)=\langle\mu,g\rangle,\qquad \mu\in\K, g\in\Cx(X).$$ 
The dual
operator $J^*$ is $\bnu\mapsto \bnu^{\#}$. 
Define
$$
\aligned
U&=\{J\rho\colon\rho\in \Cx(X),\;\|\rho\|_1\le \xi\}
,\\
V&= \{u\in C(\K)\colon u>1 \}.
\endaligned
$$
Then $U,V$ are convex and $V$ is open. 
Further,
$U\cap V=\emptyset$ by the definition
of $M_{1,c}^A$: if 
$\rho\in C_0(X)$ and 
$J\rho\ge \lambda>1$, then $\rho/\lambda$ is admissible
and thus $\|\rho\|_1\ge\lambda M_{1,c}^A(\K)>\xi$.
By the Hahn-Banach theorem there exist $a\in\er$
and $\bnu\in\M(\K)$ such that 
$\bnu\le a$ on $U$ and $\bnu>a$ on $V$. 
We claim that $a>0$. Of course, $a\ge 0$ as $U$ contains $0$.

Further, there is an $\meas$-integrable strictly positive function $h\in\Cx(X)$.
(This is a stronger version of Lemma~\ref{L:aprox}(i).) Since $X$ is Polish and $\meas$ is locally finite, there is a strictly positive continuous $m$-integrable function $h_0$. The existence of $h$ may
be now proved by applying the proof of Lemma~\ref{L:aprox}(i) to function $h_0$ in place of $1$.
%Indeed, for each $x\in X$ we may find an open set $G_x\subset X$ containing $x$ such that $\meas(G_x)<\infty$. By property (A2) there is a function $g_x\in \Cx(X)^+$ such that $g_x(x)=1$ and $g_x|_{X\setminus G_x}=0$. The sets $$G'_x=\{y\in X\colon g_x(y)>0\}, \quad x\in X$$ form an open cover of $X$. Hence there is a countable subcover, i.e., a countable set $C\subset X$ such that $X=\bigcup_{x\in C} G'_x$. Since each $g_x$ is $\meas$-integrable, we may find strictly positive coefficients $a_x$, $x\in C$ such that $$h=\sum_{x\in C} a_x g_x$$ is $\meas$-integrable. It is also strictly positive and belongs to $\Cx(X)$.

We have $Jh>0$ on $\K$ as $\K$ does not contain
the null measure. By compactness of $\K$, $Jh$ attains a strictly positive
minimum in $\K$, and thus a positive multiple of $Jh$ belongs to $V$.
Therefore $\langle\bnu,Jh\rangle>0$ and since another positive multiple of $Jh$ belongs
to $U$, we have $a>0$.
Now we can normalize to get $a=1$, so that we may assume that
$\bnu\le 1$ on $U$ and $\bnu>1$ on $V$. 
Let $u\in \Cx(\K)^+$. Then for each $\ep>0$ we have $2+u/\ep\in V$ and thus
$$
\langle\bnu,u+2\ep\rangle=\ep\left\langle\bnu,2 +\frac{u}{\ep}\right\rangle>\ep.
$$
Letting $\ep\to 0$ we obtain  $\langle\bnu,u\rangle\ge0$. By property (A4) and dominated convergence
theorem we deduce that $\langle\bnu,f\rangle\ge 0$ for each $f\in C_b(X)^+$, i.e., $\bnu\in \M^+(X)$.
We claim that $\bnu$ is a probability
measure. Indeed, given $\ep>0$, we have
$$
1+\ep\in V\implies  \langle\bnu,\,1+\ep\rangle >1.
$$
On the other hand, we find an admissible $g\in \Cx(X)$ such that
$\|g\|_1\le \xi(1+\ep)$. Then
$(1+\ep)^{-1}Jg\in U$ and $Jg\ge 1$ on $\K$. Thus
$$
\langle\bnu,\,1\rangle\le\langle\bnu,\,Jg\rangle\le 1+\ep 
$$
and we  have verified that $\bnu$ is a probability measure.  Next,
the inequality $\bnu\le 1$ on $U$ implies that for any $g\in \Cx(X)$ with $\|g\|_1\le 1$
we have
$$
 \langle \bnu^\#,g\rangle=\langle J^*\bnu,g\rangle=\langle \bnu,Jg\rangle
=\frac1\xi=\langle \bnu,J(\xi g)\rangle\le \frac1\xi.
$$
It follows that
\eqn{estimate}
$$|\langle \bnu^\#,g\rangle|\le\frac1\xi\cdot \|g\|_1,\qquad g\in \Cx(X).$$
From this inequality we first deduce that $\bnu^\#\ll \meas$. Let $N\subset X$ be an $\meas$-null set. 
Then, given $\ep>0$ there is an open set $G\supset X$ with $\meas(G)<\ep$. By Lemma~\ref{L:aprox}(ii) there is a sequence $(g_n)$ in $\Cx(X)^+$ with $g_n\nearrow \chi_G$. Then
$$\bnu^\#(G)=\lim_n \langle \bnu^\#,g_n\rangle \le \limsup_n \frac1\xi\cdot\|g_n\|_1\le \frac1\xi\cdot\meas(G)<\frac{\ep}{\xi}.$$
Thus $\bnu^\#(N)=0$.

Once we know that $\bnu^\#\ll\meas$, \eqref{estimate} shows that $\|\bnu^\#\|_\infty\le \frac1\xi$.
Thus,
$$
M_{1,c}^A(\K)=\xi=\xi\bnu(\M^+(X))\le \C(\K).
$$
\end{proof}

\begin{rmrk}
Note that the proof of $AM(\K)=M_1(\K)$ is a miracle.
One would expect a construction of a single admissible function 
$\rho$ from an admissible sequence $\rho_j$ using a suitable
covering of the compact set $\K$. However, the sets that would be useful
for such a proof  
are not open. So, instead of this we use a non-constructive proof using the Hahn-Banach theorem.
\end{rmrk}

\begin{prop}\label{p:AMcomp} 
Let  $(\K_k)_k$ be an increasing sequence of $\taux$-compact  
subsets of $M^+(X)$. 
Then
\eqn{incrforc}
$$
AM\Bigl(\bigcup_k\K_k\Bigr)=\lim_k AM(\K_k).
$$
\end{prop}

\begin{proof}
We obtain \eqref{incrforc}  from Corollary \ref{C2} and Theorem
\ref{t:compok}.
\end{proof}

\begin{prop}\label{L:wupper} The moduli $AM_p$, $M_p$ and $M_{p,c}^A$ satisfy condition \eqref{wupper} and thus \eqref{downward} of Definition \ref{d:capacity} for each $p\ge 1$.
\end{prop}

\begin{proof} 
Let $\K\subset\M^+(X)$
be a $\taux$-compact set, $\ep>0$ and let $\rho\in \Cx(X)$ be an admissible function 
for $\K$, then the set 
$$
\G=\bigl\{\mu\in \M^+(X)\colon \langle \mu,\rho\rangle >(1+\ep)^{-1}\bigr\}
$$
is $\taux$-open in $\M^+(X)$ and $(1+\ep)\rho$ is admissible for $M_{p,c}^A(\G)$.
Hence
$$
AM_p(\G)\le M_p(\G)\le M_{p,c}^A(\G)\le (1+\ep)M_{p,c}^A(\K)=(1+\ep)M_{p}(\K)=(1+\ep)AM_p(\K),
$$
where the inequalities are obvious, the first equality follows from Theorem~\ref{t:eq}
and the second one from Theorem \ref{t:compok} if $p=1$ and from Remark \ref{r:reflexive}
if $p>1$.
\end{proof}

In view of Theorem \ref{t:their}, the following result is a counterpart of Theorem \ref{t:choquet} 
for $p=1$.

\begin{thm}\label{t:choquetness} 
$\C$ is a Choquet capacity on $(\M^+(X),\tau_b)$.
In fact, $\C$ satisfies conditions (i)--(iii) from Definition~\ref{d:capacity} also on $(\M^+(X),\taux)$
if $\Cx$ is an acceptable subspace of $C_b(X)$.
\end{thm}

\begin{proof}
Property (i) is obvious, (ii) follows from Theorem \ref{t:incr} (and does not depend on the topology).
It remains to verify (iii). In fact, we can even verify (v).

Let $\K\subset\M^+(X)$ be a $\taux$-compact set and $c>\C(\K)$. By Theorem~\ref{t:compok}
we deduce that $c>AM(\K)$. Futher, by Proposition~\ref{L:wupper} there is a $\taux$-open set $\G\supset\K$ with $AM(\G)<c$. By by Proposition~\ref{p:cplem} we deduce that $\C(\G)<c$ and the proof is complete.
\end{proof}

\begin{rmrk} 
Since $\C$ is a Choquet capacity and $AM$ not, we cannot
expect the equality $AM=\C$ in general. It is even worse.
In view of Theorems \ref{t:incr} and 
\ref{t:nonincr-meas} we see that there may exist
a Borel set $\E\subset \M^+(X)$ which
fails to be $AM$-capacitable. 
Indeed, consider the sets $\E_j$ and $\E$ as in Lemma~\ref{l:construction}. Then $\E$ is a Borel set, but
\eqn{evenworse}
$$
\aligned
\sup\{AM(\K)\colon &\K\subset\E\;\tau_0\text{-compact}\}
=\sup\{\C(\K)\colon \K\subset\E\;\tau_0\text{-compact}\}
\\&=\C(\E)
=\lim_j\C(\E_j)\le
\lim_jAM(\E_j)< AM(\E).
\endaligned
$$
\end{rmrk}

\section{Modulus and content on locally compact spaces}\label{sec:lc}

In this section we assume that $X$ is
a locally compact Polish space and $\Cx(X)=C_0(X)$. In this seting we obtain a better correspondence 
between $AM$-modulus and $1$-content; we can express the $AM$-modulus in terms of the $1$-content.
The advantage of local compactness is that $\M(X)$ is the dual space to $C_0(X)$, see Section~\ref{sec:topologies}.

It follows that, for each $c>0$, the set $\{\mu\in\M^+(X)\colon \mu(X)\le c\}$
is bounded and weak* closed
in $C_0(X)^*$, hence it is weak* compact, or, in our notation, $\tau_0$ compact.
Consequently, $\M^+(X)$ is $\sigma$-compact in $\tau_0$ and thus
each $\F_{\sigma}$-subset of $\M^+(X)$ is a countable union of $\tau_0$-compact sets.

We obtain the following final result on modulus and content.

\begin{thm}\label{t:main}
Let $\E\subset \M^+(X)$ be arbitrary. Then
\eqn{main}
$$
\C(\E)\le AM(\E)=\inf\{\C(\G)\colon \G\supset \E \text{ is }\F_{\sigma} \text{ in }\tau_0\}
$$
\end{thm}

\begin{proof}
By Lemma \ref{l:efsigma}, we have
$$
AM(\E)=\inf\{AM(\G)\colon \G\supset \E \text{ is }\F_{\sigma} \}.
$$
However, if $\G\subset \M^+(X)$ is $\F_{\sigma}$ in $\tau_0$, there exists
an increasing sequence $(\K_j)_j$ of $\tau_0$-compact sets such that 
$\G=\bigcup_j\K_j$.
By Proposition \ref{p:AMcomp}, Theorem \ref{t:compok}, and Theorem \ref{t:incr},
$$
AM(\G)=\lim_{j\to\infty}AM(\K_j)=\lim_{j\to\infty}\C(\K_j)=\C(\G)
$$
for each $\taux$-open set $\G$.
Hence
$$
\C(\E)\le \inf\{\C(\G)\colon \G\supset \E \text{ is }\F_{\sigma} \text{ in }\tau_0\}
=\inf\{\AM(\G)\colon \G\supset \E \text{ is }\F_{\sigma} \text{ in }\tau_0\}=\AM(\E).
$$
\end{proof}

\begin{rmrk}
The equality in \eqref{main} resembles usual definitions of outer capacities from ``precapacities''
on compact set.
Unexpectedly, we consider infimum over $\F_{\sigma}$ sets and not over open sets.
The reason is that $AM$ is not outer regular, see Theorem \ref{t:nonouter}.
\end{rmrk}

\begin{rmrk} Let $\E$ be as in Lemma~\ref{l:construction}. Then by \eqref{evenworse} and 
Theorem \ref{t:main},
$$
\C(\E)<AM(\E)=\inf\{\C(\G)\colon \G\supset \E \text{ is }\F_{\sigma} \text{ in }\tau_0\}.
$$
So, $\C$ fails not only outer regularity, but also the weaker form of upper regularity
which is satisfied by $AM$ by Lemma \ref{l:efsigma}. Thus, from the point of view of 
continuity on increasing families, $\C$ is ``better'' than $AM$, but $\C$ has 
worse regularity properties.
\end{rmrk}

%\begin{comment}

%\end{comment}

%\bibliographystyle{abbrv}
%\bibliography{plan}

\begin{thebibliography}{10}

\bibitem{AB}
L.~Ahlfors and A.~Beurling.
\newblock Conformal invariants and function-theoretic null-sets.
\newblock {\em Acta Math.}, 83:101--129, 1950.

\bibitem{AE}
H.~Aikawa and M.~Ess\'{e}n.
\newblock {\em Potential theory---selected topics}, volume 1633 of {\em Lecture
  Notes in Mathematics}.
\newblock Springer-Verlag, Berlin, 1996.

\bibitem{ADM}
L.~Ambrosio and S.~Di~Marino.
\newblock Equivalent definitions of {$BV$} space and of total variation on
  metric measure spaces.
\newblock {\em J. Funct. Anal.}, 266(7):4150--4188, 2014.

\bibitem{ADS}
L.~Ambrosio, S.~Di~Marino, and G.~Savar\'{e}.
\newblock On the duality between {$p$}-modulus and probability measures.
\newblock {\em J. Eur. Math. Soc. (JEMS)}, 17(8):1817--1853, 2015.

\bibitem{AGS}
L.~Ambrosio, N.~Gigli, and G.~Savar\'{e}.
\newblock Density of {L}ipschitz functions and equivalence of weak gradients in
  metric measure spaces.
\newblock {\em Rev. Mat. Iberoam.}, 29(3):969--996, 2013.

\bibitem{AGS1}
L.~Ambrosio, N.~Gigli, and G.~Savar\'{e}.
\newblock Calculus and heat flow in metric measure spaces and applications to
  spaces with {R}icci bounds from below.
\newblock {\em Invent. Math.}, 195(2):289--391, 2014.

\bibitem{Choq}
G.~Choquet.
\newblock Theory of capacities.
\newblock {\em Ann. Inst. Fourier (Grenoble)}, 5:131--295 (1955), 1953/54.

\bibitem{Dav}
R.~O. Davies.
\newblock Increasing sequences of sets and {H}ausdorff measure.
\newblock {\em Proc. London Math. Soc. (3)}, 20:222--236, 1970.

\bibitem{DCEBKS}
E.~Durand-Cartagena, S.~Eriksson-Bique, R.~Korte, and N.~Shanmugalingam.
\newblock Equivalence of two {BV} classes of functions in metric spaces, and
  existence of a {S}emmes family of curves under a 1-{P}oincar\'e inequality.
\newblock Preprint arXiv:1809.03861, 2018.

\bibitem{Eng}
R.~Engelking.
\newblock {\em General topology}, volume~6 of {\em Sigma Series in Pure
  Mathematics}.
\newblock Heldermann Verlag, Berlin, second edition, 1989.
\newblock Translated from the Polish by the author.

\bibitem{Fed}
H.~Federer.
\newblock {\em Geometric measure theory}.
\newblock Die Grundlehren der mathematischen Wissenschaften, Band 153.
  Springer-Verlag New York Inc., New York, 1969.

\bibitem{fremlin}
D.~H. Fremlin.
\newblock {\em Measure theory. {V}ol. 4}.
\newblock Torres Fremlin, Colchester, 2006.
\newblock Topological measure spaces. Part I, II, Corrected second printing of
  the 2003 original.

\bibitem{Fug}
B.~Fuglede.
\newblock Extremal length and functional completion.
\newblock {\em Acta Math.}, 98:171--219, 1957.

\bibitem{HoKa}
P.~Holick\'{y} and O.~Kalenda.
\newblock Descriptive properties of spaces of measures.
\newblock {\em Bull. Polish Acad. Sci. Math.}, 47(1):37--51, 1999.

\bibitem{HEMM2}
V.~Honzlov\'{a}~Exnerov\'{a}, J.~Mal\'{y}, and O.~Martio.
\newblock Functions of bounded variation and the {$AM$}-modulus in {$\Bbb
  R^n$}.
\newblock {\em Nonlinear Anal.}, 177(part B):553--571, 2018.

\bibitem{kechris}
A.~S. Kechris.
\newblock {\em Classical descriptive set theory}, volume 156 of {\em Graduate
  Texts in Mathematics}.
\newblock Springer-Verlag, New York, 1995.

\bibitem{LMZ}
J.~Luke{\v{s}}, J.~Mal{\'y}, and L.~Zaj{\'{\i}}{\v{c}}ek.
\newblock {\em Fine topology methods in real analysis and potential theory},
  volume 1189 of {\em Lecture Notes in Mathematics}.
\newblock Springer-Verlag, Berlin, 1986.

\bibitem{M1}
O.~Martio.
\newblock Functions of bounded variation and curves in metric measure spaces.
\newblock {\em Adv. Calc. Var.}, 9(4):305--322, 2016.

\bibitem{MRSY}
O.~Martio, V.~Ryazanov, U.~Srebro, and E.~Yakubov.
\newblock {\em Moduli in modern mapping theory}.
\newblock Springer Monographs in Mathematics. Springer, New York, 2009.

\bibitem{Mi}
M.~Miranda, Jr.
\newblock Functions of bounded variation on ``good'' metric spaces.
\newblock {\em J. Math. Pures Appl. (9)}, 82(8):975--1004, 2003.

\bibitem{rudin}
W.~Rudin.
\newblock {\em Functional analysis}.
\newblock International Series in Pure and Applied Mathematics. McGraw-Hill,
  Inc., New York, second edition, 1991.

\bibitem{Sav}
G.~Savar\'e.
\newblock {S}obolev spaces in extended metric-measure spaces.
\newblock Preprint arXiv:1911.04321, 2020.

\bibitem{Schwartz}
L.~Schwartz.
\newblock {\em Radon measures on arbitrary topological spaces and cylindrical
  measures}.
\newblock Published for the Tata Institute of Fundamental Research, Bombay by
  Oxford University Press, London, 1973.
\newblock Tata Institute of Fundamental Research Studies in Mathematics, No. 6.

\bibitem{Topsoe}
F.~Tops\o~e.
\newblock {\em Topology and measure}.
\newblock Lecture Notes in Mathematics, Vol. 133. Springer-Verlag, Berlin-New
  York, 1970.

\bibitem{Zie}
W.~P. Ziemer.
\newblock Extremal length and {$p$}-capacity.
\newblock {\em Michigan Math. J.}, 16:43--51, 1969.

\end{thebibliography}

\end{document}